\documentclass[10pt, a4paper, twoside]{article}

\usepackage{amsmath}
\usepackage{amsfonts}
\usepackage{amssymb}
\usepackage{amsthm}
\usepackage{verbatim}
\usepackage[all]{xy}

\DeclareSymbolFont{rsfs}{U}{rsfs}{m}{n}
\DeclareSymbolFontAlphabet{\mathcal}{rsfs}

\DeclareFontEncoding{OT2}{}{} 
\DeclareTextFontCommand{\textcyr}{\fontencoding{OT2} \fontfamily{wncyr}\fontseries{m}\fontshape{n}\selectfont}
\newcommand{\Sha}{\textcyr{Sh}}

\newcommand{\capbar}{\overline}

\theoremstyle{plain}

\newtheorem{theorem}{Theorem}[section]
\newtheorem{maintheorem}[theorem]{Main Theorem}
\newtheorem{proposition}[theorem]{Proposition}
\newtheorem{lemma}[theorem]{Lemma}
\newtheorem{corollary}[theorem]{Corollary}

\theoremstyle{definition}
\newtheorem{subsec}[theorem]{}
\newtheorem{remark}[theorem]{Remark}
\newtheorem{construction}[theorem]{Construction}

\def \kbar {{\overline{k}}}

\def \Romannumeral #1 {\expandafter\uppercase\expandafter {\romannumeral #1} }

\def \Br {{\rm{Br}}}

\def \A{{\mathbf A}}

\def \tor {{\rm{tor}}}
\def \Spec {{\rm{Spec\,}}}

\def \Hom {{\rm {Hom}}}

\def \GL {{\rm {GL}}}
\def \SL {{\rm {SL}}}

\def \PGL {{\rm {PGL}}}

\def\ov{\overline}
\def\Q{{\mathbf{Q}}}
\def\Z{{\mathbf{Z}}}

\def\Gm{{\mathbb{G}_m}}

\def\Gal{{\rm Gal}}

\def\Hbar{{\overline{H}}}

\def\lra{\longrightarrow}

\newcommand{\sV}{\mathcal{V}}

\newcommand{\isoto}{\overset{\sim}{\to}}

\newcommand{\into}{\hookrightarrow}
\newcommand{\onto}{\twoheadrightarrow}
\newcommand{\labelto}[1]{\xrightarrow{\makebox[1.5em]{\scriptsize ${#1}$}}}

\DeclareMathOperator{\coker}{coker}

\def\uu{^\mathrm{u}}

\def\red{^\mathrm{red}}
\def\tor{^{\mathrm{tor}}}
\def\sc{^{\mathrm{sc}}}
\def\sss{^{\mathrm{ss}}}
\def\scu{^{\mathrm{scu}}}

\def\ssu{^{\mathrm{ssu}}}

\def\tors{_{\mathrm{tors}}}

\def\sab{^{\mathrm{sab}}}
\def\lin{^{\mathrm{lin}}}

\DeclareMathOperator{\U}{U}

\newcommand{\Xbar}{{\capbar{X}}}
\newcommand{\Ybar}{{\capbar{Y}}}

\newcommand{\Gbar}{{\capbar{G}}}

\def\id{{\rm id}}

\def\et{{\textup{\'et}}}

\def\sO{\mathcal{O}}

\def\gg{{\mathfrak{g}}}

\def\pibar{\overline{\pi}}

\def\Bro{\textup{Br}_1}

\def\A{{\mathbf{A}}}
\def\U{{\mathcal{U}}}

\def\R{{\mathbf{R}}}

\renewcommand{\sV}{{\mathcal{V}}}
\newcommand{\sU}{{\mathcal{U}}}
\newcommand{\abvar}{{}^{\textup{abvar}}}
\newcommand{\ab}{{\textup{ab}}}
\newcommand{\kk}{{k}}

\def\P{P${}^S$}

\newcommand{\Bra}{{{\rm Br}_{\rm a}}}

\def\Br{{\rm Br}}
\def\Pic{{\rm Pic}}

\newcommand{\Ov}{{\sO_v}}

\def\X{{\mathcal{X}}}
\def\Y{{\mathcal{Y}}}
\def\N{{\mathcal{N}}}
\def\G{{\mathcal{G}}}

\def\wt{\widetilde}

\def\Phat{{\widehat{P}}}
\def\im{\rm im}

\def\pn{\par\noindent}
\def\C{{\bf C}}

\renewcommand{\cdot}{.}

\begin{document}

\title{Manin obstruction to strong approximation\\ for homogeneous spaces}

\author{Mikhail Borovoi\thanks{The first-named author was partially supported
by the Israel Science Foundation (grant No. 807/07)
and by the Hermann Minkowski Center for Geometry}
\\Raymond and Beverly Sackler School of Mathematical Sciences,\\
Tel Aviv University, 69978 Tel Aviv, Israel\\borovoi@post.tau.ac.il
\and
Cyril Demarche\\
Universit\'e Paris-Sud, Laboratoire de Math\'ematiques d'Orsay,\\
91405 Orsay Cedex, France
\\cyril.demarche@math.u-psud.fr}

\date{\today}

\maketitle

\begin{abstract}
For a homogeneous space $X$ (not necessarily principal)
of a connected algebraic group $G$ (not necessarily linear)
over a number field $k$, we prove a theorem
of strong approximation for the adelic points of $X$
in the Brauer-Manin set.
Namely, for an adelic point $x$ of $X$
orthogonal to a certain subgroup (which may contain transcendental elements)
of the Brauer group $\Br(X)$ of $X$
with respect to the Manin pairing,
we prove a strong approximation property for  $x$
away from a finite set $S$ of places of $k$.
Our result extends a result of Harari
for torsors of semiabelian varieties
and a result of Colliot-Th\'el\`ene and Xu
for homogeneous spaces of simply connected semisimple groups,
and our proof uses those results.
\end{abstract}
\bigskip

\pn
\bigskip

\pn

\setcounter{section}{-1}



\section{Introduction}\label{sec:introduction}

Let $k$ be a number field.
We denote by $\Omega$ the set of all places of $k$,
by $\Omega_\infty$ the set of all infinite (archimedean) places of $k$,
by $\Omega_r$ the set of all real places of $k$, and by $\Omega_f$ the set of all finite (nonarchimedean) places of $k$.
For a finite set $S\subset\Omega$, we set $k_S:=\prod_{v\in S} k_v$,
where $k_v$ denotes the completion of $k$ at $v$.
We write $k_\infty$ for $k_{\Omega_\infty}$.
We denote by $\A$ the ring of ad\`eles of $k$ and by $\A^S$
the ring of ad\`eles without $S$-components.
We have $\A=\A^S\times k_S$.
If $S = \Omega_{\infty}$, we denote by $\A^f := \A^{\Omega_\infty}$ the ring of finite adeles.
If $X$ is  a $k$-variety, we have $X(k_S)=\prod_{v\in S} X(k_v)$
and $X(\A)=X(\A^S)\times X(k_S)$.
In particular $X(\A)=X(\A^f)\times X(k_\infty)$.

Let $X$ be a smooth geometrically integral $k$-variety
over a field $k$ of
characteristic 0.
Let $\Br(X):=H^2_{\textup{\'et}}(X,\Gm)$ denote the cohomological Brauer group of $X$.
We set $\Bro(X):=\ker [\Br(X)\to\Br(X\times_k \kbar)]$,
where $\kbar$ is an algebraic closure of $k$.

Recall that when $k$ is  a number field, there exists a canonical pairing
(Manin pairing)
\begin{equation}\label{eq:Manin-pairing}
\Br(X)\times X(\A)\to \Q/\Z,\quad b, x\mapsto \langle b,x\rangle,
\end{equation}
see \cite{CTSa}, Section 3.1,  or \cite{Sk}, Section 5.2.
This pairing is additive in $b\in\Br(X)$ and continuous in $x\in X(\A)$.
If $x\in X(k)\subset X(\A)$ or  $b$ comes from $\Br(k)$, then $\langle b,x\rangle=0$.

For a subgroup $B\subset\Br(X)$ we denote by $X(\A)^B$ the set of points
of $x\in X(\A)$ orthogonal to $B$ with respect to Manin pairing.
We have
$$
X(k)\subset X(\A)^{\Br(X)}\subset X(\A)^B.
$$
One can ask whether any point $x=(x_v)\in X(\A)$ which is orthogonal to $\Br(X)$
can be approximated in a certain sense by $k$-rational points.

In this paper
we consider the case when $X$ is a homogeneous space
of a connected algebraic $k$-group $G$ (not necessarily linear)
with connected geometric stabilizers.
For such  an $X$ and $x=(x_v)\in X(\A)^{\Bro(X)}$ 
it was proved in \cite{BCTS}, Appendix, Theorem A.1,  that our $X$ has a $k$-point
and that $x$ can be approximated by $k$-points in the sense of
weak approximation.
We used a result of Harari \cite{H1} on the Manin obstruction
to  weak approximation
for principal homogeneous spaces of semiabelian varieties.
Here, using a result of recent Harari's paper \cite{H}
on the Manin obstruction to  {strong} approximation
for principal homogeneous spaces of semiabelian varieties
together with a recent result of Colliot-Th\'el\`ene and Xu \cite{CTX}
on strong approximation for  homogeneous spaces of simply connected groups,
we prove a theorem
on {\em strong} approximation for our $x$.

For a connected $k$-group $G$ we write $G\abvar$
for the biggest quotient of $G$ which is an abelian variety,
and we write $G\sc$ for ``the simply connected semisimple part of $G$'',
see \ref{subsec:groups} below for details.

\begin{theorem} \label{thm:intro-Br}
Let $G$ be a connected  algebraic group (not necessarily linear)
over a number field $k$.
Let $X := H \backslash G$ be a right homogeneous space of $G$,
where $H$ is a connected $k$-subgroup of $G$.
Assume that the Tate--Shafarevich group
of the maximal abelian  variety quotient $G\abvar$ of $G$ is finite.
Let $S\supset\Omega_\infty$ be a finite set of places of $k$ containing  all  archimedean places.
We assume that $G\sc(k)$ is dense in $G\sc(\A^S)$.
Let $x=(x_v)\in  X(\A)$ be a point orthogonal to
$\Br(X)$ with respect to the Manin pairing.
Then for any open neighbourhood $\sU^{S}$ of the projection $x^S$ of $x$ to $X(\A^S)$
there exists a rational point $x_0 \in X(k)$
whose diagonal image in $X(\A^S)$ lies in $\sU^{S}$.
Moreover, we can ensure  that for each archimedean place $v$,
the points $x_0$ and $x_{v}$ lie in the same connected
component of $X(k_{v})$.
\end{theorem}

Recall that $G\sc(k)$ is dense in $G\sc(\A^S)$ if and only if
for every $k$-simple factor $G_i\sc$ of $G\sc$ the group
$G\sc_i(k_S)$ is noncompact
(a theorem of Kneser and Platonov, cf. \cite{PR}, Theorem 7.12).

Theorem \ref{thm:intro-Br} extends  a  result of Harari (\cite{H}, Theorem 4)
and a result  of Colliot-Th\'el\`ene and Xu (\cite{CTX}, Theorem 3.7(b)).

In Theorem \ref{thm:intro-Br} we assume that our adelic point $x$ is orthogonal
to the whole Brauer group $\Br(X)$.
Actually it is sufficient  to require that $x$ were orthogonal to a certain subgroup $\Br_1(X,G)\subset \Br(X)$.
In general this subgroup $\Br_1(X,G)$ contains transcendental elements (i.e. is not contained in $\Bro(X)$).
Note that Theorem \ref{thm:intro-Br} with $\Bro(X)$ instead of $\Br(X)$ would be false,
see  Counter-example \ref{subsec:Bro-versus-Br} below.
However this theorem still holds with $\Bro(X)$ instead of $\Br(X)$, if  $S$ contains at least
one nonarchimedean place:

\begin{theorem} \label{thm:intro-Bro}
Let $G$ be a connected  algebraic group (not necessarily linear)
over a number field $k$.
Let $X := H \backslash G$ be a right homogeneous space of $G$,
where $H$ is a connected $k$-subgroup of $G$.
Assume that the Tate--Shafarevich group
of the maximal abelian  variety quotient $G\abvar$ of $G$ is finite.
Let $S\supset\Omega_\infty$ be a finite set of places of $k$
containing  all  archimedean places and {\em at least one nonarchimedean place.}
We assume that $G\sc(k)$ is dense in $G\sc(\A^S)$.
Let $x=(x_v)\in  X(\A)$ be a point orthogonal to
$\Bro(X)$ with respect to the Manin pairing.
Then for any open neighbourhood $\sU^{S}$ of the projection $x^S$ of $x$ to $X(\A^S)$
there exists a rational point $x_0 \in X(k)$
whose diagonal image in $X(\A^S)$ lies in $\sU^{S}$.
Moreover, we can ensure  that for each archimedean place $v$,
the points $x_0$ and $x_{v}$ lie in the same connected
component of $X(k_{v})$.
\end{theorem}

\begin{remark}
 Let $X$ be  a right homogeneous space  of a connected group $k$-group $G$ over a number field $k$
such that the stabilizers of the geometric points are connected.
By \cite{BCTS}, Theorem A.1,
if there exists $x \in X(\A)$ which is orthogonal to $\Bro(X)$, then the variety $X$ must have a $k$-point,
hence $X=H\backslash G$, where $H$ is a connected subgroup of $G$.
Therefore we could reformulate Theorems \ref{thm:intro-Br}, \ref{thm:intro-Bro}, \ref{thm:Harari-generalized} and \ref{thm:Main-algebraic} for a general homogeneous space $X$ of $G$ (without assuming that $X$ is of the form $X=H\backslash G$).
\end{remark}

Our proof is somewhat similar to that of Theorem A.1 of \cite{BCTS}.
We use the reductions and constructions
of Subsections 3.1 and 3.3 of \cite{BCTS} in order to reduce the assertion
to the case when $X$ is a $k$-torsor under  a semiabelian variety
(treated by Harari \cite{H})
and to the case when $X$ is a homogeneous space of a simply connected semisimple group
with connected geometric stabilizers (the Hasse principle was proved in \cite{B93};
for strong approximation, see Theorem \ref{prop:simply-connected-homogeneous-spaces} below,
which is actually due to Colliot-Th\'el\`ene and Xu  \cite{CTX}).



\section{Main results}\label{section:Main-results}

\begin{subsec}\label{subsec:groups}
Let $G$ be a connected algebraic group (not necessarily linear)
over a field $k$ of characteristic 0.
Then $G$ fits into a canonical short exact sequence
$$
1\to G\lin\to G\to G\abvar\to 1,
$$
where $G\lin$ is a connected linear $k$-group and
$G\abvar$ is an abelian variety over $k$.
We use the following notation:
\pn $G\uu$ is the unipotent radical of $G\lin$;
\pn $G\red:=G\lin/G\uu$, it is a reductive $k$-group;
\pn $G\sss$ is the commutator subgroup of $G\red$,
it is a semisimple $k$-group;
\pn $G\sc$ is the universal covering of $G\sss$,
it is a simply connected semisimple $k$-group;
\pn $G\tor:=G\red/G\sss$, it is a $k$-torus;
\pn $G\ssu:=\textup{Ker}(G\lin \to G\tor)$, it is an extension of
$G\sss$ by $G\uu$.
\pn $G\sab:=(G/G\uu)/G\sss$, it is a semiabelian variety over $k$,
it fits into a short exact sequence
$$
1\to G\tor\to G\sab\to G\abvar\to 1.
$$
We define  the group $G\scu$ as the fibre product $G\scu := G\sc \times_{G\red} G\lin$,
it fits into an  exact sequence
$$
1 \to G\uu \to G\scu \to G\sc \to 1.
$$
We have a canonical homomorphism $G\scu\to G\lin\to G$.
\end{subsec}

\begin{subsec}\label{subsec:Br-def}
Let $X$ be a smooth geometrically integral $k$-variety.
We write $\Xbar$ for $X\times_k \kbar$,
where $\kbar$ is a fixed algebraic closure of $k$.
Recall that $\Br(X)$ is the cohomological Brauer group of $X$
and that $\Bro(X)=\ker [\Br(X)\to\Br(\Xbar)]$.
We set $\Bra(X):=\coker[\Br(k)\to\Bro(X)]$.
If $x_0\in X(k)$ is a $k$-point of $X$, we set
$\Br_{x_0}(X):=\ker[x_0^*\colon \Br(X)\to\Br(k)]$ and
$\Br_{1,x_0}(X):=\ker[x_0^*\colon \Bro(X)\to\Br(k)]$.
We have a canonical isomorphism $\Br_{1,x_0}(X)\isoto \Bra(X)$.

Let $X=H\backslash G$ be a homogeneous space of a connected $k$-group $G$.
Let $x_0\in X(k)$.
Consider the map $\pi_{x_0}\colon G\to X,\ g\mapsto x_0\cdot g$, it
induces a homomorphism $\pi_{x_0}^*\colon \Br(X)\to\Br(G)$.
Consider the commutative diagram
$$
\xymatrix{
{\Br(X)}\ar[r]^{\pi_{x_0}^*} \ar[d]   &\Br(G)\ar[d]\\
\Br(\Xbar)\ar[r]^{\pi_{x_0}^*}        &\Br(\Gbar) \, .
}
$$
Let $\Br_1(X,G)$ denote the kernel of any of the two
equal composed homomorphisms $\Br(X) \to \Br(\Gbar)$.
In other words, $\Br_1(X,G)$ is the subgroup of elements $b\in\Br(X)$
such that $\pi_{x_0}^*(b)\in\Br_{1}(G)$.
We show in Lemma \ref{lem:Br1(X,G)} below, that $\Br_1(X,G)$ does not depend on $x_0$.
Let
$$
\Br_{1,x_0}(X,G):=\ker[ x_0^*\colon \Br_1(X,G)\to\Br(k)]
=\{b\in\Br(X)\ |\ \pi_{x_0}^*(b)\in\Br_{1,e}(G)\}.
$$
It is easy to see that the structure map $X\to\Spec(k)$
induces an embedding $\Br(k)\into \Br_1(X,G)$
(because $X$ has a $k$-point)
and that $\Br_1(X,G)=\Br_{1,x_0}(X,G) + \Br(k)$.
It follows that an adelic point $x\in X(\A)$ is orthogonal to $\Br_1(X,G)$
with respect to Manin pairing
if and only if it is orthogonal to $\Br_{1,x_0}(X,G)$.
\end{subsec}

\begin{subsec}\label{subsec:bullet}
Let $X$ be a smooth geometrically integral $k$-variety over a number field $k$.
We denote by $X(\A)_{\bullet}$ 
the set $X(\A^f) \times \prod_{v \in \Omega_\infty} \pi_0(X(k_v))$,
where $\pi_0(X(k_v))$ is the set of connected components of $X(k_v)$.
The set $X(\A)_{\bullet}$ has a natural topology, which we call the adelic topology.
We have a canonical continuous map $X(\A)\to X(\A)_{\bullet}$
and a canonical embedding $X(k)\into X(\A)_{\bullet}$.
The pairing \eqref{eq:Manin-pairing} of the Introduction induces a pairing
\begin{equation}\label{eq:Manin-pairing-0}
\Br(X)\times X(\A)_{\bullet}\to \Q/\Z.
\end{equation}
For a subgroup $B\subset\Br(X)$ we denote by $(X(\A)_{\bullet})^B$ the set of points
of $x\in X(\A)_{\bullet}$ orthogonal to $B$ with respect to Manin pairing.

Let $X$ be a homogeneous space of a connected $k$-group $G$.
Then $G(\A)$ acts on $X(\A)$ and on $X(\A)_{\bullet}$.
\end{subsec}

\begin{maintheorem} \label{thm:Harari-generalized}
Let $G$ be a connected  algebraic group (not necessarily linear)
over a number field $k$.
Let $X := H \backslash G$ be a right homogeneous space of $G$,
where $H$ is a connected $k$-subgroup of $G$.
Assume that the Tate--Shafarevich group
of the maximal abelian  variety quotient $G\abvar$ of $G$ is finite.
Let $S\supset\Omega_\infty$ be a finite set of places of $k$
containing  all  archimedean places.
We assume that $G\sc(k)$ is dense in $G\sc(\A^S)$.
Set $S_f:=S\cap\Omega_f=S\smallsetminus \Omega_\infty$.
Then the set $(X(\A)_{\bullet})^{\Br_1(X,G)}$ coincides with the closure of the set
$X(k)\cdot G\scu(k_{S_f})$ in $X(\A)_{\bullet}$ for the adelic topology.
\end{maintheorem}

\begin{remark}
If $x\in X(\A)_{\bullet}$ is \emph{not} orthogonal to $\Bro(X,G)$,
we regard it as an obstruction to strong approximation for $x$
(this is {\em the Manin obstruction to strong approximation} from the title of the paper).
Indeed, then by the trivial part of Main Theorem \ref{thm:Harari-generalized}
the point $x$ does not belong to the closure of $X(k)\cdot G\scu(k_{S_f})$.
We interpret the nontrivial part of this theorem as an assertion that under certain assumptions
the Manin obstruction is {\em the only obstruction to strong approximation} for $x$ and for its projection $x^S\in X(\A^S)$.
Indeed, if there is no Manin obstruction, i.e. $x$ is orthogonal to $\Bro(X,G)$,
then by the nontrivial part of Main Theorem \ref{thm:Harari-generalized}
the point $x$ belongs to the closure of $X(k)\cdot G\scu(k_{S_f})$,
and its projection $x^S\in X(\A^S)$ belongs to the closure of $X(k)$ in $X(\A^S)$.
\end{remark}

\begin{subsec}\label{subsec:Bro-versus-Br}
{\em Counter-example with $\Bro(X)$ instead of $\Br(X)$.}
Above we stated the main theorem about the Manin obstruction
related to the subgroup $\Br_1(X,G)\subset\Br(X)$.
We are interested also in  {\em the algebraic Manin obstruction},
that is, the obstruction coming from the subgroup $\Bro(X) \subset \Br(X)$.
We can easily see that
in general the algebraic Manin obstruction is not the only obstruction to strong approximation.

First we notice that in general $\Bro(X)\subsetneqq \Bro(X,G)$.
Indeed, assume that $G$ is semisimple and simply connected and that $H$ is connected semisimple but not simply connected.
In particular, $H$ fits into an exact sequence
$$
1 \to \mu_H \to H\sc \to H \to 1,
$$
where $\mu_H$ is finite and abelian. Consider $X := H \backslash G$.
Then we know by \cite{San}, Proposition 6.9(iv), that the groups $\Pic(G)$ and $\Br_{1,e}(G)$ are trivial.
By \cite{CTX} Proposition 2.10(ii) we have a canonical isomorphism
$\Pic(H) \cong \Br_{x_0}(X)$, where $x_0 \in X(k)$ is the image of $e\in G(k)$.
Since $H$ is semisimple,
by the exact sequence in \cite{San}, Lemma 6.9(i),  we know that the map $\Pic(H) \to \Pic(\Hbar)$ is injective,
therefore any non-trivial element of $\Br_{x_0}(X)$ is a transcendental element of $\Br(X)$, i.e. is not killed in $\Br(\Xbar)$.
We see that $\Br_{1,x_0}(X)=0$, hence $\Bra(X)=0$.
In addition, Proposition 2.6(iii) in \cite{CTX} (or the corollary in the introduction of \cite{G})
implies that $\Br(\Gbar) = 0$, hence
$\Br_{1,x_0}(X,G) = \Br_{x_0}(X)$.
Therefore in this case $\Br_1(X)/\Br(k)=\Bra(X) = 0$,
while $\Br_1(X,G)/\Br(k)\cong\Br_{1,x_0}(X,G) \cong \Pic(H)\cong \widehat{\mu_H}(k)$.
We see that $\Br_1(X,G)\not\subset\Br_1(X)$ if $\widehat{\mu_H}(k)\neq 0$.

\def\eps{{\varepsilon}}
\def\T{T}
\def\TT{U}
An explicit example is given by $H = {\rm SO}_n \subset  \SL_n =G$  for $n \geq 3$.
In this case $\mu_H=\mu_2$ and $\widehat{\mu_H}(k)=\Z/2\Z\neq 0$, hence $\Br_1(X,G)\not\subset\Br_1(X)$.
We take  $k$ to be a totally imaginary number field, e.g. $k=\Q(i)$.
Take $S=\Omega_\infty$, then $\A^S=\A^f$.
We show that in this case the algebraic Manin obstruction
is not the only obstruction to strong approximation away from $S$.

We may and shall identify $X:=H\backslash G$ with the variety of symmetric $n\times n$-matrices $\T$ with determinant 1.
Then an element of $X(\A)$ can be written as $(\T_v)_{v\in\Omega}$, where $T_v$ is a symmetric $n\times n$-matrix
with determinant 1 over $k_v$.
Let $\eps_v(\T_v)\in\{\pm 1 \}$ denote the Hasse invariant of the
quadratic form defined by $\T_v$ (see for instance \cite{OM}, p.167, before Example 6:14).
Note that for $v\in\Omega_\infty$ we have $k_v\cong \C$, hence $\eps_v(\T_v)=1$.
For $\T^f\in X(\A^f)$  set
$$
\eps^f(\T^f)=\prod_{v\in\Omega_f} \eps_v(\T_v),
$$
then $\eps^f$ is a continuous function on $X(\A^f)$ with values $\pm 1$.
For $\T_0\in X(k)$ we have
$$
\eps^f(\T_0)=\prod_{v\in\Omega_f} \eps_v(\T_0)=\prod_{v\in\Omega_f} \eps_v(\T_0)\prod_{v\in\Omega_\infty} \eps_v(\T_0)
=\prod_{v\in\Omega}\eps_v(\T_0)  =1,
$$
because for all $v\in\Omega_\infty$ we have $\eps_v(\T_0)=1$,
and because by  \cite{OM}, Theorem 72:1, the last product equals 1.
Since $\eps^f$ is a continuous function,
for any $\T^f=(\T_v)\in X(\A^f)$ lying in the closure of $X(k)$
we have $\eps^f(\T^f)=1$.
We show below that there exists $\TT^f\in X(\A^f)$ with $\eps^f(\TT^f)=-1$.

Fix $v_0\in\Omega_f$.
Let $\TT_{v_0}\in X(k_{v_0})$ be a symmetric matrix with determinant 1 with $\eps_{v_0}(\TT_{v_0})=-1$
(there exists such a symmetric matrix, see  \cite{OM}, Theorem 63:22).
For $v\in\Omega_f\smallsetminus\{v_0\}$ set $\TT_v:={\rm diag}(1,\dots,1)\in X(k_v)$,
then $\eps_v(\TT_v)=1$.
We obtain an element $\TT^f=(\TT_v)_{v\in\Omega_f}\in X(\A^f)$ with $\eps^f(\TT^f)=-1$.
We see that $\TT^f$ does not lie in the closure of $X(k)$ in $X(\A^f)$.

It is well known that for our $G=\SL_n$ and $S=\Omega_\infty$,
the  group $G(k)$ is dense in $G(\A^S)=G(\A^f)$.
We have $\Bro(X)=\Br(k)$, hence \emph{all} the points of $X(\A)_{\bullet}=X(\A^f)$ are orthogonal to $\Bro(X)$,
in particular our point $\TT^f$.
However,  $\TT^f$ does not lie in the closure of $X(k)$ in $X(\A^f)=X(\A^S)=X(\A)_{\bullet}$.

Of course, our $\TT^f$  is not orthogonal to $\Br_1(X,G)$
(otherwise it would lie in the closure of $X(k)$ by Main Theorem \ref{thm:Harari-generalized}).
Indeed, consider the map
$$
m\colon X(\A^f)=X(\A)_{\bullet}\lra \Hom(\Br_1(X,G)/\Br(k),\Q/\Z)=\Z/2\Z
$$
induced by the Manin pairing \eqref{eq:Manin-pairing-0} from \ref{subsec:bullet}.
It can be shown that this map coincides with the map $\eps^f\colon X(\A^f)\to \{\pm 1\}$
under the canonical identification $\Z/2\Z\cong\{\pm 1\}$.
It follows that $m(\TT^f)=1+2\Z\in \Z/2\Z$,
hence $\TT^f$ is not orthogonal to $\Bro(X,G)$.
\end{subsec}

The above counter-example shows that Main Theorem \ref{thm:Harari-generalized}
does not hold with $\Bro(X)$ instead of $\Br_1(X,G)$. Nevertheless, we
can prove a similar  result about the algebraic Manin obstruction,
assuming that $S$ contains at least one nonarchimedean place.

\begin{theorem}\label{thm:Main-algebraic}
Let $G$ be a connected  algebraic group (not necessarily linear)
over a number field $k$.
Let $X := H \backslash G$ be a right homogeneous space of $G$,
where $H$ is a connected $k$-subgroup of $G$.
Assume that the Tate--Shafarevich group
of the maximal abelian  variety quotient $G\abvar$ of $G$ is finite.
Let $x\in  X(\A)$ be an adelic point orthogonal to
$\Bro X$ with respect to the Manin pairing.
Let $S\supset\Omega_\infty$ be a finite set of places of $k$
containing  all  archimedean places and {\em at least one nonarchimedean place} $v_0$.
We assume that $G\sc(k)$ is dense in $G\sc(\A^S)$.
Set $S'_f := S \smallsetminus(\Omega_\infty\cup \{ v_0 \})$.
We write  $X(\A^{\{v_0\}})_{\bullet}$ for $X(\A^{\Omega_\infty\cup\{v_0\}}) \times \prod_{v \in \Omega_\infty} \pi_0(X(k_v))$.
Then the projection $x^{\{v_0\}} \in X(\A^{\{v_0\}})_{\bullet}$ of $x$
lies in the closure of the set $X(k) \cdot G\scu(k_{S'_f})$ in $X(\A^{\{v_0\}})_{\bullet}$ for the adelic topology.
\end{theorem}

Theorem \ref{thm:intro-Br} follows from Theorem \ref{thm:Harari-generalized},
and Theorem \ref{thm:intro-Bro} follows from Theorem \ref{thm:Main-algebraic}.



\section{Sansuc's exact sequence}\label{sec:Sansuc}

\begin{subsec}
Let $k$ be a field of characteristic 0.

Let $\pi \colon Y \to X$ be a (left) torsor under a connected linear $k$-group $H$.
We define $\Br_1(X,Y)$ to be the following group:
$$
\Br_1(X,Y) := \left\{ b \in \Br(X) : \pi^*(b) \in \Br_{1}(Y) \right\}
$$
and if $y \in Y(k)$, $x=\pi(y)$, define $\Br_{1,x}(X,Y)$ to be
$$
\Br_{1,x}(X,Y) :=\ker[x^*\colon \Br_1(X,Y)\to\Br(k)]
=  \left\{ b \in \Br(X) : \pi^*(b) \in \Br_{1,y}(Y) \right\}.
$$

We denote by
$$
\langle, \rangle\colon\  \Br(X)\times X(k)\to \Br(k)\colon\
(b,x)\mapsto b(x)
$$
the evaluation map.
\end{subsec}

\begin{subsec}
Before recalling the result of Sansuc, we give  a few more definitions
and notations. Let $\mathcal{A}$ be an abelian category and $F \colon \underline{\textup{Var}}/k \rightarrow
\mathcal{A}$ be a contravariant functor from the category of $k$-varieties to
$\mathcal{A}$. If $X$ and $Y$ are $k$-varieties, the projections
$p_X, p_Y \colon X \times_k Y \rightarrow X, Y$ induce a morphism in
$\mathcal{A}$ (see \cite{San}, Section 6.b):
$$F(p_X) + F(p_Y) \colon F(X) \oplus F(Y) \to F(X \times_k Y)$$
such that
\begin{equation}
\label{def sum}
F(p_X) + F(p_Y) = F(p_X) \circ \pi_X + F(p_Y) \circ \pi_Y \, ,
\end{equation}
where $\pi_X, \pi_Y$ are the projections $F(X) \oplus F(Y) \to
F(X), F(Y)$ and the group law in the right-hand side is the law in
$\textup{Hom}(F(X) \oplus F(Y), F(X \times_k Y))$.

Let $m \colon X \times_k Y \to Y$ be a morphism of $k$-varieties.
Assume that the morphism $F(p_X) + F(p_Y)$ is an isomorphism.
We define a map
$$\varphi : F(Y) \to F(X\times_k Y)\to F(X)\oplus F(Y)  \to F(X)$$
by the formula
\begin{equation}
\label{defi varphi}
\varphi := \pi_X \circ (F(p_X) + F(p_Y))^{-1} \circ F(m)
\end{equation}
(see \cite{San}, (6.4.1)).
\end{subsec}

\begin{lemma}
\label{lem compatibility functors}
Let $F : \underline{\textup{Var}}/k \rightarrow
\mathcal{A}$ be a contravariant functor. Let $X$, $Y$ be two
$k$-varieties, $m : X \times_k Y \to Y$ be a $k$-morphism.
Assume that:
\begin{itemize}
  \item $F(\textup{Spec}(k)) = 0$.
    \item  $F(p_X) + F(p_Y)  \colon\  F(X) \oplus F(Y) \to F(X \times_k Y) $ is an isomorphism.
      \item There exists $x \in X(k)$ such that the morphism $m(x,.) : Y \to Y$ is the identity of $Y$.
\end{itemize}
Then $F(m) = F(p_X) \circ \varphi + F(p_Y) \colon\  F(Y)\to F(X\times_k Y) $.
\end{lemma}

\begin{proof}
Consider the morphism $x_Y\colon Y \to X \times_k Y$ defined by $x$. Then $F(x_Y) \circ
F(p_X) = 0$, since the morphism $p_X \circ x_Y\colon Y \to X$ factors
through $x : \textup{Spec}(k) \rightarrow X$ and $F(\textup{Spec}(k))
= 0$. Since $p_Y \circ x_Y = \textup{id}_Y$, we have $F(x_Y) \circ F(p_Y) =
\textup{id}$, and the third assumption of the lemma implies that $F(x_Y) \circ F(m) =
\textup{id}$. Therefore, we deduce that
$$F(x_Y) \circ (F(p_X) + F(p_Y)) = \pi_Y \colon\  F(X)\oplus F(Y)\to F(Y), $$
hence
\begin{equation}
\label{functors eq-2}
\pi_Y \circ (F(p_X) + F(p_Y))^{-1} \circ F(m) = F(x_Y) \circ F(m) =
\textup{id}_{F(Y)} \, .
\end{equation}
But by \eqref{def sum} we have
\begin{align*}
F(m) = F(p_X) \circ \pi_X \circ (F(p_X) + &F(p_Y))^{-1} \circ F(m) \\
+&F(p_Y) \circ \pi_Y \circ (F(p_X) + F(p_Y))^{-1} \circ F(m)\, ,
\end{align*}
so \eqref{defi varphi} and \eqref{functors eq-2} give exactly
$$F(m) = F(p_X) \circ \varphi + F(p_Y) \, .$$
\end{proof}

We shall apply those constructions and this lemma to the functors $F =
\Pic(.)$ and $F = \Bra(.)$ and to the morphism $m : H \times Y \to Y$
defined by an action of an algebraic group $H$ on a variety $Y$. In
this context, those functors satisfy the assumptions of
Lemma \ref{lem compatibility functors} by \cite{San}, Lemma 6.6.

We now recall Sansuc's result.

\begin{proposition}[{\rm \cite{San}, Proposition 6.10}]
\label{prop:Sansuc-original}
 Let $k$ be a field of characteristic zero, $H$ a connected linear $k$-group, $X$ a smooth $k$-variety and
$\pi : Y \xrightarrow{H} X$ a torsor under $H$. Then we have a functorial exact sequence:
\begin{equation}
\label{eq:Sansuc-0}
\textup{Pic}(Y) \xrightarrow{\varphi_1} \textup{Pic}(H) \xrightarrow{\Delta'_{Y/X}} \textup{Br}(X) \xrightarrow{\pi^*}
\textup{Br}(Y) \xrightarrow{m^* - p_Y^*} \textup{Br}(H \times Y).
\end{equation}
Here $m : H \times Y \to Y$ denotes the left action of $H$ on $Y$, $p_Y : H \times Y \to Y$
the natural projection and $\Delta'_{Y/X} : \Pic(H) \to \Br(X)$ is a map defined in the proof
of Proposition 6.10 in \cite{San}. For the definition of the map
$\varphi_1$, see \eqref{defi varphi} or \cite{San}, (6.4.1).
\end{proposition}
\begin{proof}
For the part of the exact sequence up to $\Br(X)\labelto{\pi^*}\Br(Y)$
see \cite{San}, Proposition 6.10.
Concerning the last map $\Br(Y) \to  \Br(H \times Y)$, it implicitly  appears
in Sansuc's paper in the last term of the exact sequence
$$
0\to\check{H}^1(Y/X,\mathcal{P}ic)\to H^2(X,\Gm)\xrightarrow{p} \check{H}^0(Y/X,\mathcal{B}r')\to 0
$$
(the second sequence of the four short exact sequences on page 45).
By definition
$$
\check{H}^0(Y/X,\mathcal{B}r')=\ker[\Br(Y) \xrightarrow{\textup{pr}_1^* - \textup{pr}_2^*} \Br(Y \times_X Y)],
$$
where $\textup{pr}_i : Y \times_X Y \to Y$ denote the two
projections.
Since $Y\to X$ is a torsor, we have a canonical isomorphism
$H \times Y \to Y \times_X Y$ defined by $(h,y) \mapsto (m(h,y),y)$.
Sansuc noticed that the maps $\textup{pr}_1$ and $\textup{pr}_2$ correspond under this isomorphism
to the maps $m$ and $p_Y$, respectively (see \cite{San}, the formulas for the faces of the simplicial system
on page 44 before Lemma 6.12).
Thus we see that
$$
\check{H}^0(Y/X,\mathcal{B}r')=\ker[\Br(Y) \xrightarrow{m^* - p_Y^*}
\textup{Br}(H \times Y)] \, .
$$
A computation using the \v{C}ech spectral sequence (6.12.0) in
\cite{San} shows that the map $\Br(X)\to\Br(Y)$ defined by the composition
$$H^2(X, \Gm) \xrightarrow{p} \check{H}^0(Y/X,\mathcal{B}r') \subset H^2(Y,\Gm)$$
is the pullback morphism $\pi^* : \Br(X)\to\Br(Y)$.
This concludes the proof of the exactness of \eqref{eq:Sansuc-0}.
\end{proof}

Such an exact sequence will be very useful in the following, but we
need another exact sequence:  we need a version
of this exact sequence with the map
$\Delta_{Y/X} : \Pic(H)\to \Br(X)$,
defined in \cite{CTX} before Proposition 2.3.
We recall here the definition of $\Delta_{Y/X}$ due to Colliot-Th\'el\`ene and Xu.

\begin{subsec}\label{subsec:definiton-of-DelataX}
{\em Definition of $\Delta_{Y/X}$.}
We use the above notation.
Since $H$ is connected, we have a canonical isomorphism
$c_H : \textup{Ext}^c_k(H,\Gm) \cong \textup{Pic}(H)$ (see \cite{CT}, Corollary 5.7),
where  $\textup{Ext}^c_k(H,\Gm)$  is the abelian group of isomorphism
classes of central extensions of $k$-algebraic groups of $H$ by
$\Gm$. Given such an extension
$$1 \to \Gm \to H_1 \to H \to 1$$
corresponding to an element $p \in \Pic(H)$, we get a coboundary map in \'etale cohomology
$$\partial_{H_1} : H^1(X, H) \to H^2(X, \Gm),$$
see \cite{Gir}, IV.4.2.2.
This coboundary map fits in the natural exact sequence of pointed sets (see
\cite{Gir}, Remark IV.4.2.10 )
\begin{equation}\label{nonabelian-exact}
 H^1(X,H_1)\to H^1(X,H)\labelto{\partial_{H_1}}H^2(X,\Gm)=\Br(X).
\end{equation}
The element $\Delta_{Y/X}(p)$ is defined to be the image of the class
$[Y] \in H^1(X,H)$ of the torsor $\pi : Y \to X$ by the map
$\partial_{H_1}$. This construction defines a map
$$
\Delta_{Y/X} \colon \Pic(H) \to \Br(X),\quad p\mapsto \partial_{H_1}([Y]),
$$
which is functorial in $X$ and $H$
(this map was denoted  by $\delta_{\textup{tors}}(Y)$ in \cite{CTX}).

We can compare the map $\Delta_{Y/X} : \Pic(H) \to \Br(X)$ with another useful
map $\alpha_{Y/X} : H^1(k, \widehat{H}) \to \Br_1(X)$ defined by the formula
$$\alpha_{Y/X}(z) := p_X^*(z) \cup [Y] \in H^2(X, \Gm) \, ,$$
where $p_X : X \to \textup{Spec}(k)$ is the structure morphism and $[Y]
\in H^1(X,H)$.

Recall that we have a canonical map $\eta_H : H^1(k, \widehat{H})
\rightarrow \Pic(H)$ coming from Leray's spectral sequence (see for
instance \cite{San}, Lemma 6.9).

\begin{lemma}
  The following diagram is commutative :
\begin{displaymath}
\xymatrix{
H^1(k, \widehat{H}) \ar[r]^{\alpha_{Y/X}} \ar[d]^{\eta_H} & \Br_1(X) \ar[d] \\
\Pic(X) \ar[r]^{\Delta_{Y/X}} & \Br(X) \, .
}
\end{displaymath}
\end{lemma}

\begin{proof}
  Define $Z$ to be the quotient of $Y$ by the action of $H\ssu$. Then
  $Z \to X$ is a torsor under $H\tor$. By functoriality, and using the
  isomorphism $\widehat{H\tor} \cong \widehat{H}$, it is sufficient to
  prove the commutativity of the following diagram :
\begin{displaymath}
\xymatrix{
H^1(k, \widehat{H\tor}) \ar[r]^{\alpha_{Z/X}} \ar[d]^{\eta_{H\tor}} & \Br_1(X) \ar[d] \\
\Pic(X) \ar[r]^{\Delta_{Z/X}} & \Br(X) \, .
}
\end{displaymath}

Consider the groups $\textup{Ext}^n_k(H\tor, \Gm)$ in the
abelian category of fppf-sheaves over $\textup{Spec}(k)$.
By \cite{Ka}, Lemmas A.3.1 and A.3.2, we know that the diagram
\begin{displaymath}
\xymatrix{
H^1(k, \widehat{H\tor}) \ar[r]^{\alpha_{Z/X}} \ar[d]^{\eta'_{H\tor}} & \Br_1(X)
\ar[d] \\
\textup{Ext}^1_k(H\tor, \Gm) \ar[r]^(.6){\Delta_{Z/X}} & \Br(X)
}
\end{displaymath}
is commutative, where $\eta'_{H\tor} : H^1(k, \widehat{H\tor}) \rightarrow
\textup{Ext}^1_k(H\tor, \Gm)$ is the edge map from the local to global
$\textup{Ext}$'s spectral sequence
$H^p(k, \textup{Ext}^q_{\overline{k}}(\ov{H}\tor,
\Gm_{\ov{k}})) \implies \textup{Ext}_k^{p+q}(H\tor, \Gm)$ (see
\cite{SGA4}, V.6.1).

By \cite{Oo}, Proposition 17.5, there exists a canonical map
$\textup{Ext}^1_k(H\tor, \Gm) \rightarrow \textup{Ext}^c_k(H\tor,
\Gm)$. Composing this map with $c_{H\tor} : \textup{Ext}^c_k(H\tor,
\Gm) \to \Pic(H\tor)$ (used in the construction of
$\Delta_{Z/X}$), we get a map $c'_{H\tor} :
\textup{Ext}^1_k(H\tor, \Gm) \rightarrow \Pic(H\tor)$.

It is now sufficient to prove that the diagram
\begin{displaymath}
\xymatrix{
H^1(k, \widehat{H\tor}) \ar[rd]_{\eta_{H\tor}} \ar[r]^(.4){\eta'_{H\tor}} &
\textup{Ext}_k^1(H\tor, \Gm) \ar[d]^{c'_{H\tor}} \\
& \Pic(H\tor)
}
\end{displaymath}
is commutative.

The natural transformation
$\textup{Hom}_{\overline{k}-\textup{groups}}(\ov{H}\tor,(.)_{\ov{k}}) \xrightarrow{\tau}
H^0(\ov{H}\tor, (.)_{\ov{k}})$
of functors from the category of fppf-sheaves over $\textup{Spec}(k)$ to the category of
$\textup{Gal}(\overline{k}/k)$-modules
induces a morphism of spectral
sequences
\begin{align*}
( H^p(k, \textup{Ext}^q_{\overline{k}}(\ov{H}\tor,
\Gm_{\ov{k}})) \implies &\textup{Ext}_k^{p+q}(H\tor, \Gm))    \labelto{\tau}\\
&( H^p(k, H^q(\ov{H}\tor, \Gm_{\ov{k}})) \implies
  H^{p+q}(H\tor, \Gm) )
\end{align*}
from the local to global $\textup{Ext}$'s spectral sequence to
Leray's spectral sequence. This morphism implies that the induced
diagram between edge maps
\begin{displaymath}
\xymatrix{
H^1(k, \widehat{H\tor}) \ar[r]^(.4){\eta'_{H\tor}} \ar[rd]_{\eta_{H\tor}} & \textup{Ext}_k^1(H\tor, \Gm)
\ar[d]^{\tau_0} \\
& \Pic(H\tor)
}
\end{displaymath}
is commutative.
We need to prove that the map $\tau_0$ induced by
$\tau$ coincides with the map $c'_{H\tor} :
\textup{Ext}_1^c(H\tor, \Gm) \to \Pic(H\tor)$. Let
$$0 \to \Gm_k \to I \to Q \to 0$$
be an exact sequence of fppf-sheaves on $\textup{Spec}(k)$ such that
$I$ is injective. Then the long exact sequences associated to the functors
$H^0(H\tor, .)$ and $\textup{Hom}_k(H\tor, .)$ give rise to the
following commutative diagram
\begin{displaymath}
\xymatrix{
\textup{Hom}_k(H\tor, Q) / \textup{Hom}_k(A,I) \ar[r] \ar[d]^{\cong} &
H^0(H\tor, Q) / H^0(H\tor, I) \ar[d]^{\cong} \\
\textup{Ext}_k^1(H\tor, \Gm) \ar[r]^{\tau_0} & H^1(H\tor, \Gm) \, ,
}
\end{displaymath}
where the vertical maps are the coboundary maps and the horizontal
ones are induced by $\tau$. With this diagram, it is clear that the
image by $\tau_0$ of a given group extension is the same as the image
of this extension by the map $c'_{H\tor}$, which concludes the proof.
\end{proof}

\begin{remark}
In particular, if $H = T$ is a $k$-torus, then the map $\eta_T : H^1(k, \widehat{T})
\xrightarrow{\cong} \Pic(T)$ is an isomorphism, and we see that the
map $\Delta_{Y/X}$ coincides with the map $\alpha_{Y/X} : H^1(k, \widehat{T})
\to \Br_1(X)$ defined by $z \mapsto p_X^*(z) \cup [Y]$.
\end{remark}
\end{subsec}

The goal of the following theorem is to give an equivalent of
Sansuc's exact sequence (\ref{eq:Sansuc-0}) with the map $\Delta'_{Y/X}$
replaced by the map $\Delta_{Y/X}$. It will be very important in the following.

\begin{theorem}\label{lem Sansuc}
Let $k$ be a field of characteristic zero. Let $H$ a connected linear $k$-group, $X$ a smooth $k$-variety
and $\pi : Y \xrightarrow{H} X$ a (left) torsor under $H$.
Then we have a commutative diagram with exact rows, functorial in $(X,Y,\pi,H)$:
\begin{equation}\label{eq:Sansuc-1}
\xymatrix{
\textup{Pic}(Y) \ar[r]^{\varphi_1} &\textup{Pic}(H) \ar[r]^{\Delta_{Y/X}} &\textup{Br}(X) \ar[r]^{\pi^*}
&\textup{Br}(Y) \ar[r]^-{m^* - p_Y^*} &\textup{Br}(H \times Y)\\
\textup{Pic}(Y) \ar[r]^{\varphi_1}\ar@{=}[u] &\textup{Pic}(H) \ar[r]^-{\Delta_{Y/X}}\ar@{=}[u]
&\textup{Br}_1(X,Y) \ar[r]^-{\pi^*}\ar[u]_{\iota_X}
&\Br_1(Y) \ar[r]^{\varphi_2}\ar[u]_{\iota_Y}    &\Bra(H)\ar[u]_{\nu} \, .
}
\end{equation}
Here $m : H \times Y \to Y$ denotes the left action of $H$ on $Y$,
the homomorphism $\Delta_{Y/X} : \Pic(H) \to \Br(X)$ is the map of \cite{CTX},
see  \ref{subsec:definiton-of-DelataX} above,
the homomorphisms $\varphi_1$ and $\varphi_2$ are defined in
\cite{San} (6.4.1) (or see \eqref{defi varphi}), the homomorphisms $\iota_X$ and $\iota_Y$ are the inclusion maps,
and the injective homomorphism $\nu$ is given as the composite of the following natural injective maps:
$$
\Bra(H)\to\Bra(H)\oplus\Bro(Y)\labelto{\cong}\Bro(H\times Y)\into \Br(H\times Y).
$$

In particular, if $Y(k) \neq \emptyset$ and $y\in Y(k)$, $x=\pi(y)$,
then the maps $\varphi_i$ are induced by the map $i_y : H \to Y$ defined by $h \mapsto h\cdot y$,
and we have an exact sequence
\begin{equation}\label{eq:Sansuc-3}
\textup{Pic}(Y) \xrightarrow{i_y^*} \textup{Pic}(H) \xrightarrow{\Delta_{Y/X}} \textup{Br}_{1,x}(X,Y) \xrightarrow{\pi^*}
\Br_{1,y}(Y) \xrightarrow{i_y^*} \textup{Br}_{1,e}(H).
\end{equation}
\end{theorem}

\begin{remark}
Recall that the exact sequences \eqref{eq:Sansuc-1} and
\eqref{eq:Sansuc-3} can be extended to the left by
$$0 \rightarrow k[X]^*/k^* \rightarrow k[Y]^*/k^* \rightarrow \widehat{H}(k) \rightarrow
\Pic(X) \rightarrow \Pic(Y)$$
(see \cite{San}, Prop. 6.10).
\end{remark}

\begin{corollary}
Let $k$ be a field of characteristic zero. Let $T$ be a $k$-torus, $X$ a smooth $k$-variety
and $\pi : Y \xrightarrow{T} X$ a (left) torsor under $T$.
Then we have an exact sequence :
$$ \textup{Pic}(Y) \xrightarrow{\varphi_1} \textup{Pic}(T)
\xrightarrow{\Delta_{Y/X}} \textup{Br}_1(X) \xrightarrow{\pi^*} \Br_1(Y)
\xrightarrow{\varphi_2} \Bra(T) \, .$$
\end{corollary}

\begin{proof}
It is a direct application of Theorem \ref{lem Sansuc}, using
$\Pic(\ov{T}) = 0$.
\end{proof}

\begin{remark}
This corollary compares the algebraic Brauer groups of $X$ and
$Y$. Concerning the transcendental part of those groups, Theorem
\ref{lem Sansuc} can be used to study the injectivity of the map
$\Br(X) \rightarrow \Br(Y)$. We cannot describe easily the image of this map
in general. However, Harari and Skorobogatov studied this map in particular
cases (universal torsors for instance): see \cite{HSk}, Theorems 1.6
and 1.7.
\end{remark}

\begin{corollary}[{\rm cf. \cite{San}, Corollary 6.11}]\label{cor:Sansuc-cor}
Let
$$
1\to G'\labelto{i} G\labelto{j} G''\to 1
$$
be an exact sequence of connected algebraic groups over a field $k$ of characteristic 0.
Assume that $G'$ is linear.
Then there is a commutative diagram with exact rows
\begin{equation}\label{eq:Sansuc-cor-3}
\xymatrix{
 \Pic(G)\ar[r]^{i^*}      &\Pic(G')\ar[r]^{\Delta_{G/G''}}      &\Br(G'')\ar[r]^{j^*}
       &\Br(G)\ar[r]^-{m^* - p_G^*}  &\Br(G' \times G)\\
\Pic(G)\ar[r]^{i^*}\ar@{=}[u] &\Pic(G')\ar[r]^-{\Delta_{G/G''}}\ar@{=}[u]  &\Br_{1,e}(G'',G)\ar[r]^-{j^*}\ar[u]_{\iota''}
   &\Br_{1,e}(G)\ar[r]^{i^*}\ar[u]_\iota   &\Br_{1,e}(G')\ar[u]_\nu \, .
}
\end{equation}
Here $p_G : G' \times G \to G$ is the projection map, the map $m : G' \times G \to G$ is defined by $m(g',g) := i(g')\cdot g$
(where the product denotes the group law in $G$), $\iota''$ and $\iota$ are the inclusion homomorphisms,
and the injective homomorphism $\nu$ is defined as in Theorem \ref{lem Sansuc}.

If the homomorphism $\Pic(\Gbar)\to\Pic(\ov{G'})$ is surjective
(e.g. when $G'$ is a $k$-torus, or when ${G'}\sss$ is simply
connected, or when all the three groups $G'$, $G$ and $G''$ are linear), then  $\Br_{1,e}(G'', G)=\Br_{1,e}(G'')$,
and we have a commutative diagram with exact rows
\begin{equation}\label{eq:Sansuc-cor-1}
\xymatrix{
 \Pic(G)\ar[r]^{i^*}      &\Pic(G')\ar[r]^{\Delta_{G/G''}}      &\Br(G'')\ar[r]^{j^*}
       &\Br(G)\ar[r]^-{m^* - p_G^*}  &\Br(G' \times G)\\
\Pic(G)\ar[r]^{i^*}\ar@{=}[u] &\Pic(G')\ar[r]^{\Delta_{G/G''}}\ar@{=}[u]  &\Br_{1,e}(G'')\ar[r]^{j^*}\ar[u]_{\iota''}
   &\Br_{1,e}(G)\ar[r]^{i^*}\ar[u]_\iota   &\Br_{1,e}(G')\ar[u]_\nu \, .
}
\end{equation}
\end{corollary}

\begin{proof}[Proof of the corollary]
The short exact sequence of algebraic groups defines a structure of (left) $G''$-torsor
under $G'$ on $G$ ($G'$ acts on $G$ by left translations).
Now from the diagram with exact rows \eqref{eq:Sansuc-1} we obtain diagram \eqref{eq:Sansuc-cor-3},
which differs from diagram \eqref{eq:Sansuc-cor-1} by the middle term in the bottom row.

From diagram \eqref{eq:Sansuc-1} we obtain an exact sequence
\begin{equation}
\textup{Pic}(\Gbar) \xrightarrow{i^*} \textup{Pic}(\overline{G'})
\xrightarrow{\Delta_{G/G''}} \textup{Br}(\overline{G''}) \xrightarrow{j^*}
\textup{Br}(\Gbar) \, .
\end{equation}
If the homomorphism $i^*\colon \textup{Pic}(\Gbar) \to \textup{Pic}(\overline{G'})$ is surjective,
then the homomorphism $j^*\colon \textup{Br}(\overline{G''}) \to \textup{Br}(\Gbar)$ is injective,
hence $\Br_{1,e}(G'', G)=\Br_{1,e}(G'')$, and we obtain diagram \eqref{eq:Sansuc-cor-1} from diagram \eqref{eq:Sansuc-cor-3}.

If $G'$ is a $k$-torus or if ${G'}\sss$ is simply connected, then $\Pic(\overline{G'})=0$, and the homomorphism  $\Pic(\Gbar)\to\Pic(\overline{G'})$ is clearly surjective.
If all the three groups $G'$, $G$ and $G''$ are linear, then again
the homomorphism $\textup{Pic}(\Gbar) \to \textup{Pic}(\overline{G'})$ is surjective,
see \cite{San}, proof of Corollary 6.11, p. 44.
\end{proof}

For the proof of Theorem \ref{lem Sansuc} we need a crucial lemma.

\begin{lemma}\label{lem composition torsors}
Let $k$ be a field of characteristic zero. Let $H$ a connected linear $k$-group, $X$ a smooth $k$-variety
and $\pi : Y \xrightarrow{H} X$ a (left) torsor under $H$.
Let $\tau\colon Z \xrightarrow{\Gm} Y$ be a torsor under $\Gm$.
Then there exists a central extension of algebraic $k$-groups
$$1 \to \Gm \to H_1 \to H \to 1$$
and a left action $H_1 \times Z \to Z$, extending the action of $\Gm$ on $Z$ and compatible with the action of $H$ on $Y$.
This action makes $Z \to X$ into a torsor under $H_1$.
Moreover, the class of such an  extension $H_1$
in the group $\textup{Ext}^{\textup{c}}_k(H, \Gm)$ is uniquely determined,
namely
$[H_1]=\varphi_1([Z])  \in \textup{Ext}^{\textup{c}}_k(H, \Gm) =
\Pic(H)$.
\end{lemma}

\begin{remark}
In \cite{HSk2}, Harari and Skorobogatov studied this question of
composition of torsors. Their results (see Theorem 2.2 and Proposition
2.5 in \cite{HSk2}) deal with torsors under multiplicative groups and
not only under $\Gm$ as here, but they require additional assumptions
concerning the type of the torsor and on invertible functions on
the varieties. Those additional assumptions are not satisfied in our context.
\end{remark}

\begin{proof}
Let $p_H\colon   H\times Y \to H$ and $p_Y\colon  H\times Y\to Y$ denote the two projections. Let
$$1 \to \Gm \to H_1 \to H \to 1$$
be a central extension such that its class in $\textup{Ext}^{\textup{c}}_k(H, \Gm) = \Pic(H)$ is exactly $\varphi_1([Z])$.

In this setting, Lemma \ref{lem compatibility functors} implies that
\begin{equation} \label{formula Pic}
m^*[Z] =  p_H^* [H_1]+p_Y^* [Z] \, .
\end{equation}

Formula \eqref{formula Pic} means that the push-forward of the torsor
$H_1 \times Z \xrightarrow{\Gm \times \Gm} H \times Y$
by the group law homomorphism $\Gm \times \Gm \to \Gm$
is isomorphic (as a $H \times Y$-torsor under $\Gm$) to the pullback $m^* Z$
of the torsor $Z \to Y$ by the map $m : H \times Y \to Y$.
In particular, we get the following commutative diagram:
\begin{equation}
\label{diag group action}
\xymatrix{
H_1 \times Z \ar@/^1.5pc/@{-->}[rr]^{m'} \ar[r]^-{\Gm} \ar[rd]_{\Gm \times \Gm} & m^* Z \ar[r] \ar[d]^{\Gm} & Z \ar[d]^{\Gm} \\
& H \times Y \ar[r]^-m & Y
}
\end{equation}
where $m'$ is defined to be the composite of the two upper horizontal maps.
The situation is very similar to that in the proof
of Theorem 5.6 in \cite{CT}: the map $m'$ fits into  commutative diagram (\ref{diag group action})
and for all $t_1, t_2 \in \Gm$ and all $h_1 \in H_1, z \in Z$, we have
\begin{equation}
\label{compatible actions}
m'(t_1\cdot h_1, t_2\cdot z) = t_1 t_2 \cdot m'(h_1, z) \, .
\end{equation}
We want to use $m'$ to define a group action of $H_1$ on $Z$.
Formula (\ref{compatible actions}) implies that the morphism
$m'(e,.) : z \mapsto m'(e,z)$ is an automorphism of the $Y$-torsor $Z$, and we can define a map $m'' : H_1 \times Z \to Z$ to be
the composition
$$m'' := m'(e,.)^{-1} \circ m'\colon H_1\times Z\to Z.$$
Then we get a commutative diagram
\begin{equation}
\label{diag group action 2}
\xymatrix{
H_1 \times Z \ar[r]^(.6){m''} \ar@<2ex>[d] \ar@<-2ex>[d] & Z \ar[d] \\
H \times Y \ar[r]^(.6)m & Y
}
\end{equation}
where the map $m''$ still satisfies formula (\ref{compatible actions}) and now, for all $z \in Z$, we have
\begin{equation} \label{neutral element}
m''(e,z) = z.
\end{equation}
We wish to prove that $m''$ is a left group action of $H_1$ on $Z$.

Since $m\colon H\times Y\to Y$ is a left action, we have
$$
\tau(m''(h_1 h_2,z))=\tau(m''(h_1, m''(h_2,z))) \text{ for } h_1,h_2\in H_1,\ z\in Z,
$$
where $\tau \colon Z\to Y$ is the canonical map.
Since $\tau\colon Z\to Y$ is a torsor under $\Gm$,
there is a canonical map
$$
Z\times_Y Z\to\Gm, \quad (z_1,z_2)\mapsto z_1 z_2^{-1}.
$$
We obtain a morphism of $k$-varieties
$$
\varphi\colon H_1\times H_1\times Z\to \Gm\quad (h_1,h_2,z)\mapsto \varphi_z(h_1,h_2)
$$
such that
$$
m''(h_1 h_2,z)=\varphi_z(h_1,h_2)\cdot m''(h_1, m''(h_2,z))\text{ for }h_1,h_2\in H_1, z\in Z.
$$
Then \eqref{neutral element} implies that
$$
\varphi_z(h,e)=1 \text{ and } \varphi_z(e,h)=1.
$$
By Rosenlicht's lemma (see \cite{R}, Theorem 3, see also \cite{San},
Lemma 6.5), the map $\varphi$ has to be trivial,
i.e. $\varphi_z(h_1,h_2)=1$ for all $z,h_1,h_2$.
Therefore we have
\begin{equation}
\label{associative action}
m''(h_1h_2,z)=m''(h_1,m''(h_2,z)) \, .
\end{equation}

Formulas (\ref{neutral element}) and (\ref{associative action}) show  that $m''$
is a left group action of $H_1$ on $Z$.
Since $m''$ satisfies \eqref{compatible actions}, we have for $t\in \Gm,\ z\in Z$
$$
m''(te,z)=t\cdot m''(e,z)=t\cdot z,
$$
hence the action $m''$ extends the action of $\Gm$ on $Z$.
From diagram \eqref{diag group action 2} with $m''$ instead of $m'$
we see that the action $m''$ induces the action $m$ of $H$ on $Y$.

Consider the following commutative diagram (see (\ref{diag group action 2})):
\begin{equation}
\label{diag isom torsors}
\xymatrix{
H_1 \times_k Z \ar[r]^{\phi_Z} \ar[d] & Z \times_X Z \ar[d] \\
H \times_k Y \ar[r]^{\phi_Y} & Y \times_X Y \, ,
}
\end{equation}
where $\phi_Z(h_1,z) := (m''(h_1,z), z)$, $\phi_Y(h,y) := (m(h,y), y)$,
and the unnamed morphisms are the natural ones.
Since $Y \to X$ is a torsor under $H$, the morphism $\phi_Y$ is an
isomorphism.
The group $\Gm \times \Gm$ acts on $H_1 \times_k Z$ via $(t_1,t_2)\cdot(h_1,z) := (t_1 h_1, t_2\cdot z)$,
making $H_1 \times_k Z \to H \times_k Y$ into a torsor under $\Gm \times \Gm$.
We define an action of $\Gm\times\Gm$ on $Z\times_X Z$ by
$$
(t_1,t_2)\cdot(z_1,z_2) := ((t_1 t_2)\cdot z_1,t_2\cdot z_2),
$$
then $Z \times_X Z \to Y \times_X Y$ is a torsor under $\Gm \times \Gm$.
By  formula (\ref{compatible actions}) the map $\phi_Z$ in  (\ref{diag isom torsors})
is a morphism of torsors under $\Gm \times \Gm$ compatible with the isomorphism $\phi_Y$ of $k$-varieties.
Therefore the map $\phi_Z$ is an isomorphism of $k$-varieties,
which proves that the action $m''$ makes $Z \to X$ into a torsor under $H_1$.

Let us prove the uniqueness of the class of the extension $H_1$.
If $H_2$ is a central extension of $H$ that satisfies the conditions of the lemma,
then the analogues of diagram (\ref{diag group action 2}) and formula (\ref{compatible actions})
with $H_2$ instead of $H_1$ define an isomorphism of $H \times Y$-torsors under $\Gm$
between the push-forward of $H_2 \times Z$ by the morphism $\Gm \times \Gm \to \Gm$ and the torsor $m^* Z$. Therefore, we get
$$m^*[Z] = p_H^* [H_2] + p_Y^* [Z].$$
Comparing with  (\ref{formula Pic}), we see that  $ p_H^* [H_2]= p_H^* [H_1]$.
Since $p_H^* + p_Y^* : \Pic(H) \oplus \Pic(Y) \to \Pic(H \times Y)$
is an isomorphism, we see that $p_H^*\colon\Pic(H)  \to \Pic(H \times Y)$
is an embedding, hence $[H_2] = [H_1]$,
which completes the proof of Lemma \ref{lem composition torsors}.
\end{proof}

\begin{subsec}
{\em Proof of Theorem \ref{lem Sansuc}: Top row of the diagram.}
First we prove that the top  row in diagram (\ref{eq:Sansuc-1})
is a complex.
Let $p \in \Pic(Y)$ and let us prove that $\Delta_{Y/X}(\varphi_1(p)) = 0$.
Let $Z \to Y$ be a torsor under $\Gm$ such that $[Z] = p \in \Pic(Y)$. Let $p' := \varphi_1(p) \in \textup{Pic}(H)$,
and $1 \rightarrow \Gm \rightarrow H_1 \rightarrow H \rightarrow 1$ be a
central extension of $H$ by $\Gm$ corresponding to $p'$ via the
isomorphism $\textup{Ext}^c_k(H, \Gm) \cong \textup{Pic}(H)$. Then
$\Delta_{Y/X}(p')$ is equal (by definition) to $\partial_{H_1}([Y]) \in
H^2(X, \Gm)$, where $\partial_{H_1} : H^1(X, H) \rightarrow H^2(X,
\Gm)$ is the coboundary map coming from the extension $H_1$, and
$[Y]$ is the class of the torsor $Y \rightarrow X$ in $H^1(X,
H)$.

Lemma \ref{lem composition torsors} implies that the class $[Y] \in H^1(X, H)$ is in the image of the map $H^1(X,H_1) \to H^1(X,H)$.
From  exact sequence \eqref{nonabelian-exact} we see that the class $\partial_{H_1}([Y])$
is trivial in $H^2(X, \Gm)$, i.e. $\Delta_{Y/X}(p') = \partial_{H_1}([Y]) = 0 \in H^2(X, \Gm)$,
hence $\textup{Pic}(Y) \xrightarrow{j^*} \textup{Pic}(H) \xrightarrow{\Delta_{Y/X}} \textup{Br}(X)$ is a complex.

Let $p \in \textup{Pic}(H)$, and let us prove that $\pi^*(\Delta_{Y/X}(p)) = 0$.
The element $p$ corresponds to the class of an extension
$1 \rightarrow \Gm \rightarrow H_1 \rightarrow H \rightarrow 1$, and we have  $\Delta_{Y/X}(p) = \partial_{H_1}([Y])$.
We have a commutative diagram
\begin{displaymath}
\xymatrix{
H^1(X, H) \ar[r]^{\partial_{H_1}} \ar[d]^{\pi^*} & H^2(X, \Gm) \ar[d]^{\pi^*} \\
H^1(Y, H) \ar[r]^{\partial_{H_1}} &  H^2(Y, \Gm)
}
\end{displaymath}
so that $\pi^*(\Delta_{Y/X}(p)) = \pi^*(\partial_{H_1}([Y])) = \partial_{H_1}(\pi^*[Y])$.
The torsor $\pi^*[Y]$ is trivial in the set $H^1(Y,H)$,
hence $\pi^*(\Delta_{Y/X}(p)) = 0$.
Consequently the top row of diagram (\ref{eq:Sansuc-1})  is a complex.

Let us prove that the top row of diagram (\ref{eq:Sansuc-1}) is exact.
For the exactness at the term $\Br(Y)$, see Proposition \ref{prop:Sansuc-original}.
So it remains to prove the exactness of the top row at $\Pic(H)$ and at $\Br(X)$.

\bigskip

Let $p \in \textup{Pic}(H)$ be such that $\Delta_{Y/X}(p) = 0$. Such a $p$
corresponds to the class of an extension $1 \rightarrow \Gm \rightarrow H_1
\xrightarrow{q} H \rightarrow 1$ such that $\partial_{H_1}([Y]) = 0$.
From  exact sequence \eqref{nonabelian-exact} we see
that there exists an $X$-torsor $Z \xrightarrow{H_1} X$ under $H_1$,
with an $H_1$-equivariant map $Z \to Y$, making $Z \to Y$ into a torsor under $\Gm$.
Then by the uniqueness part of Lemma \ref{lem composition torsors}
we know that the class of $Z \to Y$ in $\Pic(Y)$
maps to the class of $H_1$ in $\Pic(H)$, i.e. $\varphi_1([Z]) = p$, which
proves the exactness of the top row of diagram (\ref{eq:Sansuc-1}) at $\Pic(H)$.

Let us now prove the exactness of the top row of (\ref{eq:Sansuc-1}) at $\Br(X)$.

Assume first that the $k$-variety $X$ is quasi-projective. Let $A \in \Br(X)$ such that $\pi^* A = 0 \in \Br(Y)$.
By a theorem of Gabber also proven by de Jong (see \cite{dJ}),
we know that there exists a positive integer $n$
and an $X$-torsor $Z \to X$ under $\PGL_n$ such that $- A$ is the image of the class
of $Z$ in $H^1(X,\PGL_n)$ by the coboundary map $H^1(X, \PGL_n) \xrightarrow{\partial_{\GL_n}} H^2(X, \Gm)$.
Let $W$ denote the product $Y \times_X Z$.
From the commutative diagram with exact rows
$$
\xymatrix{
H^1(X,\GL_n)\ar[d]\ar[r] &H^1(X,\PGL_n)\ar[d]\ar[r] &\Br(X)\ar[d]^{\pi^*}\\
H^1(Y,\GL_n)\ar[r]       &H^1(Y,\PGL_n)\ar[r]       &\Br(Y)
}
$$
we see that the assumption $\pi^* A = 0$ implies that the torsor $W \xrightarrow{\PGL_n} Y$
is dominated by some $Y$-torsor under $\GL_n$,
i.e. there exists a torsor $V \xrightarrow{\GL_n} Y$ and a morphism of $Y$-torsors $V \to W$
compatible with the quotient morphism $\GL_n \to \PGL_n$.
We have the following picture:
\begin{displaymath}
\xymatrix{
Y \ar[d]_H & W \ar[l]^{\PGL_n} \ar[d]^H & V \ar[l]^{\Gm} \ar@/_1.5pc/[ll]_{\GL_n} \\
X & Z \ar[l]^{\PGL_n} & \, .
}
\end{displaymath}
Since $W \to X$ is a torsor under the connected linear group $H \times \PGL_n$, we can apply Lemma \ref{lem composition torsors}
to get a central extension
\begin{equation}
\label{extension product}
1 \to \Gm \to L \to H \times \PGL_n \to 1
\end{equation}
and a structure of $X$-torsor under $L$ on $V \to X$, compatible with the action of $H \times \PGL_n$ on $W$.
In particular,
the natural injections of $H$ and $\PGL_n$ into $H \times \PGL_n$ define two central extensions obtained by pulling back the extension
(\ref{extension product}):
\begin{gather}
\label{exactLH}1 \to \Gm \to L_H \to H \to 1\\
\label{exact PGL} 1 \to \Gm \to L_{\PGL_n} \to \PGL_n \to 1 \, .
\end{gather}
Since $\PGL_n$ acts trivially on $Y$, the action of $L_{\PGL_n}$ (as a
subgroup of $L$) on $V$ defines a commutative diagram
\begin{displaymath}
\xymatrix{
L_{\PGL_n} \times V \ar[r]^-{\phi_V} \ar[d] & V \times_Y V \ar[d] \\
\PGL_n \times W \ar[r]^-{\phi_W} & W \times_Y W
}
\end{displaymath}
where $\phi_V(l,v) := (l.v, v)$, $\phi_W(p,w) := (p.w, w)$, and the
vertical maps are the natural ones.
We see easily that $\phi_V$ is an isomorphism, hence $V \to Y$ is a torsor under $L_{\PGL_n}$.
 This action of $L_{\PGL_n}$ extends the action of $\Gm$ on $V$ above $W$, and is compatible with the action of
$\PGL_n$ on $W$ above $Y$ via the extension (\ref{exact PGL}) and the
map $V \to W$. Therefore, the unicity result in Lemma \ref{lem  composition torsors}
implies that  exact sequence (\ref{exact PGL}) is equivalent to the usual extension
$1 \to \Gm \to \GL_n \to \PGL_n \to 1$, and in particular that $\partial_{L_{\PGL_n}}([Z]) = \partial_{\GL_n}([Z]) \in H^2(X, \Gm)$.

Consider the direct product of the  exact sequences \eqref{exactLH} and \eqref{exact PGL}
$$
1 \to \Gm \times \Gm \to L_H \times L_{\PGL_n} \to H \times \PGL_n \to
1 \, .
$$
Define the morphism $\mu : L_H \times L_{\PGL_n} \to L$ by $\mu(l,l') := l.l'$, where the product is taken inside the group $L$.
By definition of $L_H$ and $L_{\PGL_n}$, we see that the image of the commutator morphism
$c : L_H \times L_{\PGL_n} \to L$ defined by $c(l,l') := l.l'.l^{-1}.{l'}^{-1}$ is contained in the central subgroup $\Gm$ of $L$.
Therefore, since $c(e,l') = c(l,e) = 1$, Rosenlicht's lemma implies that $c(l,l') = 1$ for all $(l, l') \in L_H \times L_{\PGL_n}$.
 Hence the morphism $\mu$ is a group homomorphism, and
the following diagram is commutative with exact rows
\begin{displaymath}
\xymatrix{
1 \ar[r] & \Gm \times \Gm \ar[r] \ar[d]^{m} & L_H \times L_{\PGL_n} \ar[r] \ar[d]^{\mu} &  H \times \PGL_n \ar@{=}[d] \ar[r] & 1 \\
1 \ar[r] & \Gm \ar[r] & L \ar[r] &  H \times \PGL_n \ar[r] & 1 \, ,
}
\end{displaymath}
where $m \colon \Gm \times \Gm \to \Gm$ is the group law.
Therefore we get a commutative diagram of coboundary maps
\begin{displaymath}
\xymatrix{
H^1(X, H) \times H^1(X, \PGL_n) \ar[rr]^{\partial_{L_H} \times \partial_{L_{\PGL_n}}} & & H^2(X,\Gm) \times H^2(X, \Gm) \\
H^1(X, H \times \PGL_n) \ar@{=}[d] \ar[u]^{\cong} \ar[rr]^{\partial_{L_H \times L_{\PGL_n}}} &
     & H^2(X,\Gm \times \Gm) \ar[u]^{\cong} \ar[d]^m \\
H^1(X, H \times \PGL_n) \ar[rr]^{\partial_L} & & H^2(X, \Gm) \, .
}
\end{displaymath}
In particular, this diagram implies that
$$\partial_L([W]) = \partial_{L_H}([Y]) + \partial_{L_{\PGL_n}}([Z])$$
in $H^2(X,\Gm)$.
Since $W \xrightarrow{H \times \PGL_n} X$ is dominated by the $X$-torsor $V \xrightarrow{L} X$, we know that $\partial_L([W]) = 0$,
therefore the above formula implies that
$$\partial_{L_H}([Y]) = - \partial_{L_{\PGL_n}}([Z]) = A \in \Br(X)\, ,$$
i.e. $A = \Delta_{Y/X}([L_H])$, with $[L_H] \in \textup{Ext}^c_k(H, \Gm)
\cong \Pic(H)$. Hence the first row of diagram (\ref{eq:Sansuc-1}) is exact under
the assumption that $X$ is quasi-projective.

Let us now deduce the general case: $X$ is not supposed to be
quasi-projective anymore. By Nagata's theorem (see \cite{N}) we know that there
exists a proper $k$-variety $Z$ and an open immersion of $k$-varieties $X \to Z$.
By Chow's lemma (see \cite{EGA}, II.5.6.1 and II.5.6.2, or \cite{Sh}, Chapter VI, \S 2.1),
there exists a projective $k$-variety $Z'$ and a projective, surjective
birational morphism $Z' \to Z$.
Moreover, using Hironaka's resolution
of singularities (see \cite{Hir}, see also \cite{BM} and \cite{EV}), there exists a smooth
projective $k$-variety $\widetilde{Z}$ and a birational morphism
$\widetilde{Z} \to Z'$.
Define $X'$ to be the fibred
product $X' := \widetilde{Z} \times_Z X$. Then $X'$ is a open
subvariety of $\widetilde{Z}$, hence $X'$ is a smooth quasi-projective
$k$-variety, and the natural map $X' \to X$ is a birational
morphism.
Define $Y'$ to be the product $Y \times_X X'$.
By the quasi-projective case, we know that in the  commutative
diagram
\begin{equation}\label{eq:diagram-Delta}
\xymatrix{
\Pic(H) \ar[r]^{\Delta_{Y'/X'}} & \Br(X') \ar[r] & \Br(Y') \\
\Pic(H) \ar@{=}[u] \ar[r]^{\Delta_{Y/X}} & \Br(X) \ar[u] \ar[r] & \Br(Y) \ar[u]
}
\end{equation}
the first row is exact.
Since the map $X' \to X$ is a birational morphism,
we have a commutative diagram
\begin{equation*}
\xymatrix{
\Br(X)\ar[r]\ar[d] &\Br(X')\ar[d]\\
\Br(k(X))\ar[r]^\cong    &\Br(k(X')) \, ,
}
\end{equation*}
where the bottom horizontal arrow is an isomorphism.
Since both $X$ and $X'$ are smooth,
by \cite{Gr}, II, Corollary 1.8, the vertical arrows are injective.
It follows that the homomorphism $\Br(X) \to \Br(X')$ is injective.
Now a diagram chase in diagram  \eqref{eq:diagram-Delta}
proves the exactness of the second row in that diagram.
This completes the proof of the exactness of
the top row of diagram (\ref{eq:Sansuc-1}).
\qed
\end{subsec}

\begin{subsec}\label{subsec:proof-Sansuc-2}
{\em Proof of Theorem \ref{lem Sansuc}: The commutativity of the diagram and the exactness of its bottom row.}
It is clear that all the squares of diagram \eqref{eq:Sansuc-1} are commutative except maybe the rightmost square.

By Lemma \ref{lem compatibility functors}, we have for all $\beta_Y \in \Br_1(Y)$,
\begin{equation}\label{general formula Br}
m^*\beta_Y=p^*_H \varphi_2(\beta_Y)+p^*_Y\beta_Y\text{ in } \Br_1(H\times Y).
\end{equation}
This formula implies immediately that the rightmost square of diagram \eqref{eq:Sansuc-1} commutes,
hence this diagram is commutative.

Since the top row of (\ref{eq:Sansuc-1}) is exact, and the diagram (\ref{eq:Sansuc-1}) is commutative, it is easy
to conclude that the bottom row of (\ref{eq:Sansuc-1}) is exact.

\qed
\end{subsec}

\begin{subsec}
{\em Proof of Theorem \ref{lem Sansuc}: Exact sequence \eqref{eq:Sansuc-3}.}
We consider the following diagram:
\begin{equation}\label{eq:diagram-Sansuc-4}
\xymatrix{
\Pic(Y) \ar[r]^{i_y^*} \ar @{=}[d] & \Pic(H) \ar[r]^-{\Delta_{Y/X}} \ar @{=}[d] & \Br_{1,x}(X,Y) \ar[r]^-{\pi^*} \ar[d]
           & \Br_{1,y}(Y) \ar[r]^{i_y^*} \ar[d] & \Br_{1,e}(H) \ar[d]^{\cong} \\
\Pic(Y) \ar[r]^{\varphi_1} & \Pic(H) \ar[r]^-{\Delta_{Y/X}} & \Br_1(X,Y) \ar[r]^-{\pi^*} & \Br_1(Y) \ar[r]^{\varphi_2} & \Br_a(H) \, .
}
\end{equation}
In this diagram the second row is the second row of diagram \eqref{eq:Sansuc-1}, hence it is exact.
In the first row,
$\Br_{1,x}(X,Y)$ is a subgroup of $\Br_1(X,Y)$ and $\Br_{1,y}(Y)$ is a subgroup of $\Br_1(Y)$.
We easily check that all the arrows in the first row are well-defined.

This diagram is commutative: for the first and the last squares it
is a consequence of Lemma 6.4 of \cite{San}, and for the two central squares it is clear.
We know that the second row of this diagram is exact.
Since all the vertical arrows are injective, we see that the first row is a complex.
An easy diagram chase shows that the first row is also exact, i.e. sequence
\eqref{eq:Sansuc-3} is exact.
\qed
\end{subsec}



\section{Compatibility}\label{sec:compatibility}

In this section we use Theorem \ref{lem Sansuc}
to get a compatibility result between the evaluation map and the action of a linear group.
We begin with a lemma:

\begin{lemma}\label{lem:product}
Let $k$ be a field of characteristic 0 and $X$, $X'$ be two smooth $k$-varieties.
Let $\pi\colon Y \xrightarrow{H} X$ and $\pi'\colon Y' \xrightarrow{H'} X'$
be two (left) torsors under connected linear $k$-groups $H$ and $H'$.
The map $Y \times Y' \to X \times X'$ is naturally a torsor under $H \times H'$.
Assume that $\ov{Y'}$ is rational and that $Y'(k) \neq \emptyset$. Let $y'\in Y'(k)$ and $x':=\pi'(y')$.
Then  the homomorphism
$$p_X^* + p_{X'}^* : \textup{Br}_1(X, Y) \oplus \textup{Br}_{1,x'}(X', Y') \rightarrow \textup{Br}_1(X \times X', Y \times Y')$$
is well-defined and is an isomorphism.
\end{lemma}

\begin{proof}
Consider the three exact sequences associated to the torsors $Y \to X$, $Y' \to X'$ and $Y \times Y' \to X \times X'$
(see Theorem \ref{lem Sansuc}). We get the following commutative diagram with exact rows:
\begin{displaymath}
\xymatrix{
\Pic(Y) \oplus \Pic(Y') \ar[r] \ar[d]^{p_Y^* + p_{Y'}^*} &
\Pic(H) \oplus \Pic(H') \ar[r] \ar[d]^{p_H^* + p_{H'}^*} &
\Br_1(X,Y) \oplus \Br_{1,x'}(X',Y') \ar[d]^{p_X^* + p_{X'}^*}
\\
\Pic(Y \times Y') \ar[r] &
\Pic(H \times H') \ar[r] &
\Br_1(X \times X', Y \times Y')
\\
  \ar[r]  &
\Br_1(Y) \oplus \Br_{1,y'}(Y') \ar[r] \ar[d]^{p_Y^* + p_{Y'}^*} &
\Bra(H) \oplus \Br_{1, e}(H') \ar[d]^{p_H^* + p_{H'}^*}
\\
\ar[r] &
\Br_1(Y \times Y') \ar[r] &
\Bra(H \times H') \, .
}
\end{displaymath}
The two first and the two last vertical arrows are isomorphisms by \cite{San}, Lemma 6.6,
hence the five lemma implies that the central vertical arrow is an isomorphism.
\end{proof}

\begin{corollary}\label{cor:product}
 If in Lemma  \ref{lem:product} we also have $Y(k)\neq \emptyset$, $y\in Y(k)$ and $x=\pi(y)$, then the homomorphism
$$
p_X^* + p_{X'}^* : \Br_{1,x}(X, Y) \oplus \textup{Br}_{1,x'}(X', Y') \longrightarrow
\textup{Br}_{1, (x,x')}(X \times X', Y \times Y')
$$
is an isomorphism.
\end{corollary}

\begin{lemma}[Compatibility]\label{lem:compatibility}
Let $k$ be a field of characteristic 0.
Let $\pi : Y \to X$ be a (left) torsor under a connected linear $k$-group $H$.
Let $G$ be a  connected algebraic $k$-group,
acting on the right on $Y$ and $X$ such that $\pi$ is $G$-equivariant.
Assume that $Y$ is $\kbar$-rational and  $Y(k) \neq \emptyset$ and let $y_0 \in Y(k)$, $x_0=\pi(y_0)$.
Then for any  $b\in\Br_{1,x_0}(X,Y),\ x\in X(k),\ g\in G(k)$ we have
$$
b(x\cdot g)=b(x)+\pi^*(b)(y_0\cdot g) \, .
$$
\end{lemma}

\begin{proof}
We consider  two torsors: $\pi \colon Y \xrightarrow{H} X$ and $\pi'=\id \colon G \to G$.
Let $m_X\colon X \times G \rightarrow X$ and $m_Y\colon Y\times G\to Y$ denote the actions of $G$.
Since $m_Y(y,e)=y$,
we see  that if $b \in\textup{Br}_y(X, Y)$, then $m_X^* b \in
\textup{Br}_{(y,e)}(X \times G,\, Y \times G)$. By functoriality of
the evaluation map, we have a commutative diagram
\begin{displaymath}
\xymatrix{
\Br_{1,(x_0,e)}(X \times G,\, Y \times G) \times (X \times G)(k)
\ar@<10ex>[d]^= \ar[r]^(.78){\textup{ev}} & \Br(k) \ar@{=}[d] \\
\Br_{1,x_0}(X, Y) \oplus \Br_{1, e}(G) \times (X(k) \times
G(k)) \ar@<8ex>[u]^{p_X^* + p_G^*} \ar[r]^(.8){\textup{ev}'} &
\Br(k)
}
\end{displaymath}
where the first pairing is the evaluation map on $X \times G$
and the second one is the sum of the evaluation maps on $X$ and on $G$,
i.e.
$$
 \textup{ev}'((B,C), (x,g)) := B(x) + C(g) \in \textup{Br}(k)\text{
   for } B\in\Br_{1,x_0}(X,Y),\ C\in\Br_{1,e}(G) \, .
$$
By Corollary \ref{cor:product} the left vertical morphism $p_X^* + p_G^*$ is an isomorphism.
Therefore we get two natural projections:
\begin{align*}
\pi_X \colon &\Br_{1,(x_0,e)}(X \times G,\, Y \times G) \to \Br_{1,x_0}(X, Y)\text{ and }\\
\pi_G \colon &\Br_{1,(x_0,e)}(X \times G,\, Y \times G) \to \Br_{1,e}(G).
\end{align*}
Hence for all
$D \in \textup{Br}_{1,(x_0,e)}(X \times G,\ Y \times G)$, we get $D = p_X^*(\pi_X(D)) + p_G^*(\pi_G(D))$,
and so, by the commutativity of the diagram, for all $(x, g) \in X(k) \times G(k)$ we have
$$
D((x,g)) = \pi_X(D)(x) + \pi_G(D)(g).
$$
For $g = e \in G(k)$, this formula implies that
$$
D((x,e)) =
\pi_X(D)(x)$$
because $\pi_G(D)(e) = 0$.
So, for $b \in \Br_{1,x_0}(X,Y)$, $g \in G(k)$ and $x \in X(k)$, we have
$$
b(x\cdot g) = (m^* b)((x,g)) = \pi_X(m^* b)(x) + \pi_G(m^*b)(g) = (m^* b)((x,e)) + \pi_G(m^* b)(g)
$$
where we write $m$ for $m_X$.
Since $(m^* b)((x,e)) = b(m(x,e)) = b(x)$ by functoriality, we have
\begin{equation} \label{eq:compatibility1}
b(x\cdot g) = b(x) + \pi_G(m^* b)(g) \, .
\end{equation}
In the case $x=x_0 = \pi(y_0)$ we obtain
\begin{equation} \label{eq:compatibility2}
(\pi^*b)(y_0\cdot g) = b(\pi(y_0)\cdot g) = b(\pi(y_0)) + \pi_G(m^* b)(g) = \pi_G(m^* b)(g),
\end{equation}
since $b(\pi(y_0)) = b(x_0) = 0$.
Consequently, (\ref{eq:compatibility1}) and (\ref{eq:compatibility2}) give the expected formula, that is:
$$ b(x\cdot g) = b(x) + (\pi^*b)(y_0\cdot g) \, .$$
\end{proof}

\begin{corollary}\label{cor-lem-compatibility}
Let $X := H \backslash G$ be a right homogeneous space of a connected  $k$-group $G$
over a field $k$ of characteristic 0, where $H \subset G$
is a connected linear $k$-subgroup.
Then for all $b\in\Br_{1,x_0}(X,G),\ x\in X(k),\ g\in G(k)$ we have
$$
b(x\cdot g) =  b(x)+ \pi^*b(g),
$$
where
$\pi : G \to X$ is the quotient map and $x_0 = \pi(e)$.
\end{corollary}

\begin{proof}
 We take $Y=G$, $y_0=e\in G(k)$, and $x_0=\pi(y_0)=\pi(e)$ in Lemma \ref{lem:compatibility},
then $\Br_{1,x_0}(X,Y)=\Br_{1,x_0}(X,G)$ and $\pi^*b(y_0\cdot g)=\pi^*b(g)$.
\end{proof}

\begin{corollary}\label{cor:compatibility}
Let $k$ be a number field.
Let $X:= H \backslash G$ be a homogeneous space of a connected $k$-group $G$,
where $H \subset G$ is a connected linear $k$-subgroup.
Let $\langle,\rangle$ denote the Manin pairing
$$
\Br(X)\times X(\A)\to\Q/\Z.
$$
Let $b\in\Br_{1, x_0}(X,G),\ x\in X(\A),\ g\in G(\A)$.
Then
$$
\langle b,x\cdot g\rangle =\langle b,x\rangle
+\langle \pi^*(b),g\rangle,
$$
where $\pi : G \rightarrow X$ is the quotient map and $x_0 = \pi(e)$.
\end{corollary}

\begin{corollary}\label{cor:from-compatibility}
Let $k,G,H$ and $X$ be as in Corollary \ref{cor:compatibility}.
Let $\varphi\colon G'\to G$ be a homomorphism of $k$-groups,
where $G'$ is  \emph{a simply connected $k$-group.}
Let $b\in\Br_{1,x_0}(X,G),\ x\in X(\A),\ g'\in G'(\A)$.
Then
$$
\langle b,x\cdot \varphi(g')\rangle =\langle b,x\rangle.
$$
\end{corollary}

\begin{proof}
By Corollary \ref{cor:compatibility} we have
$$
\langle b,x\cdot \varphi(g')\rangle =\langle b,x\rangle
+\langle \pi^*(b),\varphi(g')\rangle.
$$
By functoriality we have
$$
\langle \pi^*(b),\varphi(g')\rangle = \langle \varphi^*\pi^*(b),g'\rangle.
$$
Since $b\in\Br_{1,x_0}(X,G)$, we have $\pi^*b\in\Br_{1,e}(G)$
and $\varphi^*\pi^*b\in\Br_{1,e}(G')=0$.
Thus $\varphi^*\pi^*b=0$, hence $\langle \pi^*(b),\varphi(g')\rangle=0$, and the corollary follows.
\end{proof}



\section{Some lemmas}\label{sec:some-lemmas}

For an abelian group $A$ we write $A^D:=\Hom(A,\Q/\Z)$.

\begin{lemma}\label{lem:surjectivity}
Let $P$ be a quasi-trivial $k$-torus
over a number field $k$.
Then the canonical map $\lambda\colon P(\A)\to\Bra(P)^D$
induced by the Manin pairing is surjective.
\end{lemma}

\begin{proof}
We have $\Bra(P)=H^2(k,\Phat)$, see \cite{San}, Lemma 6.9(ii). By \cite{San}, (8.11.2), the map
$$
\lambda\colon P(\A)\to \Bra(P)^D=H^2(k,\Phat)^D
$$
is given by the canonical pairing
$$
P(\A)\times H^2(k,\Phat)\to\Q/\Z.
$$
Consider the map $\mu$ from the Tate-Poitou exact sequence
\begin{equation}\label{eq:Tate-Poitou}
  (P(\A)_{\bullet})^\wedge\labelto{\mu}H^2(k,\Phat)^D\to H^1(k,P),
\end{equation}
 see \cite{HSz}, Theorem 5.6 or \cite{D1}, Theorem 6.3.
By  $(P(\A)_{\bullet})^\wedge$ we mean the completion of $P(\A)_{\bullet}$
for the topology of open subgroups of finite index.
Then the map $\mu$ is induced by $\lambda$.
Since $P$ is a quasi-trivial torus, we have  $H^1(k,P)=0$,
and we see from \eqref{eq:Tate-Poitou} that the map $\mu$ is surjective.
But by \cite{H}, Lemma 4, $\im\,\mu=\im\,\lambda$.
Thus $\lambda$ is surjective.
\end{proof}

\begin{lemma}\label{lem:open-map}
Let $X$ be a right homogeneous space (not necessarily principal)
 of a connected $k$-group $G$ over a number field $k$.
Let $N\subset G$ be a connected normal $k$-subgroup.
Set $Y:=X/N$, and let $\pi\colon X\to Y$ be the canonical map.
Then the induced  map $X(\A)\to Y(\A)$ is open.
\end{lemma}

Note that the geometric quotient  $X/N$ exists in the category of $k$-varieties
by \cite{B96}, Lemma 3.1.

\begin{proof}
If $v$ is a nonarchimedean place of $k$, we denote
by $\sO_v$ the ring of integers of $k_v$,
and by $\kappa_v$ the residue field of $\sO_v$.
For an $\sO_v$-scheme $Z_v$ we set
$\wt{Z}_v:=Z_v\times_\Ov \kappa_v$.

Since the morphism $\pi$ is smooth, the map
$X(k_v)\to Y(k_v)$ is open for any place $v$ of $k$.

Let $S$ be a finite set of places of $k$
containing all the archimedean places.
Write $\sO^S$ for the ring of elements of $k$
that are integral outside $S$.
Taking $S$ sufficiently large,
we can assume that $G$ and $N$ extend to smooth group schemes
$\mathcal G$ and $\mathcal N$ over ${\rm Spec}(\sO^S)$,
and that $X$  and $Y$ extend to homogeneous spaces
$\mathcal X$ of $\mathcal G$ and $\Y$ of $\G/\N$
over  ${\rm Spec}(\sO^S)$ such that
$\mathcal Y=\mathcal X/\mathcal N$.
In particular, the reduction $\wt N_v:=\N\times_{\sO^S} \kappa_v$
is connected
for $v\notin S$.

Let $v\notin S$ and let $y_v\in \Y(\Ov)$.
Set $X_{y_v}:=\X \times_\Y {\rm Spec}(\sO_v)$,
the morphisms being given by $\pi$ and  by $y_v\colon\Spec(\Ov)\to\Y$.
It is an $\Ov$-scheme.
Then its reduction $\wt{X_{y_v}}$ is a homogeneous space of the connected
$\kappa_v$-group $\wt N_v$ over the finite field $\kappa_v$.
 By Lang's theorem (\cite{L}, Theorem 2)
 $\wt{X_{y_v}}$ has a $\kappa_v$-point.
By Hensel's lemma $X_{y_v}$ has an $\Ov$-point.
This means that $y_v\in\pi(\X(\Ov))$.
Thus $\pi(\X(\Ov))=\Y(\Ov)$ for all $v\notin S$.
It follows that the map $X(\A)\to Y(\A)$ is open.
\end{proof}

\begin{subsec}
Let $X=H\backslash G$ be a homogeneous space of a connected $k$-group $G$.
Let $x_1\in X(k)$.
Consider the map $\pi_{x_1}\colon G\to X,\ g\mapsto x_1\cdot g$,
it induces a homomorphism $\pi_{x_1}^*\colon \Br(X)\to\Br(G)$.
Let $\Br_1(X,G)_{x_1}$ denote the subgroup of elements $b\in\Br(X)$
such that $\pi_{x_1}^*(b)\in\Br_{1}(G)$.
The following lemma shows that $\Br_1(X,G)_{x_1}$  does not depend on $x_1$,
so we may write $\Br_1(X,G)$ instead of $\Br_1(X,G)_{x_1}$.
Note that $\Br_1(X,G)=\Br_{1,x_1}(X,G) + \Br(k)$.
\end{subsec}

\begin{lemma}\label{lem:Br1(X,G)}
The subgroup $\Br_1(X,G)_{x_1}\subset \Br(X)$ does not depend on $x_1$.
\end{lemma}

\begin{proof}
\def\btil{{\tilde{b}}}
We have a commutative diagram
$$
\xymatrix{
\Br(X)/\Br(k)\ar[d]\ar[r]^{\pi^*_{x_1}}  &\Br(G)/\Br(k)\ar[d]\\
\Br(\Xbar)  \ar[r]^{\overline{\pi}^*_{x_1}}  &\Br(\Gbar) \, .
}
$$
We see that it suffices to prove that the kernel
of ${\overline{\pi}^*_{x_1}}$ does not depend on $x_1$.

Now if $x_2\in X(k)$ is another $k$-point, then
$x_2=x_1\cdot g$ for some $g\in G(\kbar)$,
hence
\def\pibar{{\overline{\pi}}}
$$
\pibar_{x_2}(g')=x_2\cdot g'=x_1\cdot gg'=\pibar_{x_1}(gg')=(\pibar_{x_1}\circ l_g)(g'),
$$
where $l_g$ denotes the left translation on $\Gbar$ by $g$.
Thus
$$
\pibar_{x_2}^*=l_g^*\circ\pibar_{x_1}^*\,.
$$
Since $l_g$ is an isomorphism of the underlying variety of $\Gbar$,
we see that $l_g^*\colon\Br(\Gbar)\to\Br(\Gbar)$ is an isomorphism, hence $\ker\pibar_{x_2}^*=\ker\pibar_{x_1}^*$,
which proves the lemma.
\end{proof}

\begin{lemma}\label{lem:unipotent-k}
Let $X$ be a right homogeneous space of a unipotent $k$-group $U$
over a field $k$ of characteristic 0.
Then $X(k)$ is non-empty and is one orbit of $U(k)$.
\end{lemma}

\begin{proof}
By \cite{B96}, Lemma 3.2(i), $X(k)\neq \emptyset$.
Let $x\in X(k)$, $H:={\rm Stab}(x)$,
then $H$ is unipotent, hence $H^1(k,H)=1$,
and therefore $X(k)=x\cdot U(k)$.
\end{proof}

\begin{corollary}\label{cor:unipotent-R}
Let $X$ be a right homogeneous space of a unipotent $\R$-group $U$.
Then $X(\R)$ is non-empty and connected.
\end{corollary}

\begin{proof}
Since $U(\R)$ is connected and $X(\R)=x\cdot U(\R)$, we conclude that $X(\R)$ is connected.
\end{proof}

\begin{lemma}\label{lem:unipotent-SA}
Let $G$ be a unipotent $k$-group over a number field $k$.
Let $X$ be a right homogeneous space of $G$.
Let $S\subset \Omega$ be any non-empty finite set of places.
Then $X(k)$ is  is non-empty and dense in $X(\A^S)$.
\end{lemma}

\begin{proof}
By \cite{B96}, Lemma 3.2(i), $X(k)$ is non-empty.
Let $x_0\in X(k)$, and let $H\subset G$ denote the stabilizer
of $x_0$ in $G$.
We have $X=H\backslash G$.

Set $\gg={\rm Lie}(G)$.
Since $\gg$ is a vector space and $S\neq\emptyset$,
by the classical strong approximation theorem $\gg$ is dense in $\gg\otimes_k\A^S$.
Since ${\rm char}(k)=0$, we have the exponential map
$\gg\to G$, which is an isomorphism of $k$-varieties.
We see that $G(k)$ is dense in $G(\A^S)$.
It follows that $x_0 G(k)$ is dense in $x_0 G(\A^S)$.
Since $H$ is unipotent, we have $H^1(k_v,H)=0$ for any $v\in\Omega$,
and therefore $x_0 G(k_v)=X(k_v)$ for any $v$.
It follows that  $x_0 G(\A^S)=X(\A^S)$ (we use Lang's theorem and Hensel's lemma).
Thus $x_0 G(k)$ is dense in $X(\A^S)$, and $X(k)$ is dense in $X(\A^S)$.
\end{proof}


\section{Brauer group}

We are grateful to A.N.~Skorobogatov, E.~Shustin, and T. Ekedahl
for helping us to prove Theorem \ref{thm:Brauer} below.

\begin{theorem} \label{thm:Brauer}
Let $X$ be a smooth irreducible algebraic variety
over an algebraically closed field $k$ of characteristic 0.
Let $G$ be a \emph{connected} algebraic group
(not necessarily linear) defined over $k$, acting on $X$.
Then $G(k)$ acts on $\Br(X)$ trivially.
\end{theorem}

\begin{proof}
We write $H^i$ for $H^i_{\textup{\'et}}$ (\'etale cohomology).
The Kummer exact sequence
$$
1\to\mu_n\to\Gm\labelto{n}\Gm\to 1
$$
of multiplication by $n$ gives rise to a surjective map
$$
H^2(X,\mu_n)\onto\Br(X)_n\;,
$$
where $\Br(X)_n$ denotes the group of elements of order dividing $n$ in $\Br(X)$.
Since every element of $\Br(X)$ is torsion
(because $\Br(X)$ embeds in $\Br(k(X))$, cf. \cite{Gr}, II, Corollary 1.8),
it is enough to prove the following Theorem \ref{thm:H^i}.
\end{proof}

\def\et{{\textup{\'et}}}

\begin{theorem}\label{thm:H^i}
Let $X$ be a smooth irreducible algebraic variety
over an algebraically closed field $k$ (of any characteristic).
Let $G$ be a \emph{connected} algebraic group
(not necessarily linear) defined over $k$, acting on $X$.
Let $A$ be a finite abelian group of order invertible in $k$.
Then $G$ acts on $H^i_\et(X,A)$ trivially for all $i$.
\end{theorem}

\begin{proof}[Proof in characteristic 0]
By the Lefschetz principle, we may assume that $k=\C$.
Let $g\in G(\C)$.
We must prove that $g$ acts trivially on the Betti cohomology $H^i_B(X,A)$.
Since $G$ is connected, the group $G(\C)$ is connected,
hence we can connect $g$ with the unit element $e\in G(\C)$ by a path.
We see that the automorphism of $X$
$$
g_*\colon X\to X,\quad x\mapsto x\cdot g
$$
is homotopic to the identity automorphism
$$
e_*\colon X\to X,\quad x\mapsto x.
$$
It follows that $g_*$ acts on $H^i_B(X,A)$ as $e_*$, i.e. trivially.
\end{proof}

To prove Theorem \ref{thm:H^i} in any characteristic, we need two lemmas.

\begin{lemma}\label{lem:direct-image}
Let $X,Y$ be smooth algebraic varieties
over an algebraically closed field $k$ (of any characteristic).
Let $A$ be a finite abelian group of order invertible in $k$.
Consider the projection $p_Y\colon X\times Y\to Y$.
Then the higher direct image $R^i (p_Y)_* A$ in the \'etale topology
is the pullback of the abelian group $H^i(X,A)$ considered as a sheaf
on $\textup{Spec}(k)$.
\end{lemma}

\begin{proof}
Consider the commutative diagram
$$
\xymatrix{
X\times_k Y\ar[r]^-{p_X}\ar[d]_{p_Y} &X\ar[d]^{s_X} \\
Y\ar[r]^-{s_Y}                 &\Spec(k) \, ,
}
$$
here $X\times_k Y$ is the fibre product of $X$ and $Y$
with respect to the structure morphisms $s_X$ and $s_Y$.
Clearly $R^i (s_X)_* A$ is the constant sheaf on $\Spec(k)$ with stalk $H^i(X,A)$.
By \cite{SGA}, Th. finitude, Theorem 1.9(ii),
the sheaf $R^i (p_Y)_* A$ on $Y$ is the pullback of the constant sheaf $R^i (s_X)_* A$ on $\Spec(k)$
along the morphism $s_Y\colon Y\to\Spec(k)$, which concludes the proof.
\end{proof}

Let $y\in Y(k)$. It defines a canonical morphism
$f_y\colon X\to X\times_k Y$ such that
$$
p_X\circ f_y=\id_X \text{ and } p_Y\circ f_y=y\circ s_X\colon X\to Y.
$$

\begin{lemma}\label{lem:constant-map}
Let $X,Y$ be two smooth algebraic varieties
over an algebraically closed field $k$ (of any characteristic).
Let $A$ be a finite abelian group of order invertible in $k$.
For a closed point $y\in Y$ consider the map $f_y\colon X\to X\times Y$ defined above.
If $Y$ is irreducible, then the map
$$
f_y^*\colon H^i(X\times Y,A)\to H^i(X,A)
$$
does not depend on $y$.
\end{lemma}

\begin{proof}
Let $\eta\in H^i(X\times Y, A)$.
Then $\eta$ defines a global section
$\lambda(\eta)$ of the sheaf $R^i (p_Y)_* A$
(via compatible local sections $U\mapsto \eta|_{X\times U}\in H^i(X\times U, A)$
of the corresponding presheaf, for all \'etale open subsets $U\to Y$).
By Lemma \ref{lem:direct-image} the sheaf $R^i (p_Y)_* A$ is a constant sheaf
with stalk $H^i(X,A)$.
It is easy to see that
$$
f_y^*(\eta)=\lambda(\eta)(y)\in H^i(X,A).
$$
Since $Y$ is irreducible, it is connected, hence the global section
$\lambda(\eta)$ of the constant sheaf $R^i (p_Y)_* A$  on $Y$ is constant,
and therefore $\lambda(\eta)(y)$ does not depend on $y$.
Thus $f_y^*(\eta)$ does not depend on $y$.
\end{proof}

\begin{proof}[Proof of Theorem \ref{thm:H^i} in any characteristic]
Consider the map
$$
m\colon X\times G\to X,\quad (x,g)\mapsto x\cdot g
$$
(the action). Let $\xi\in H^i(X,A)$.
Set $\eta =m^* \xi \in H^i(X\times G, A)$. For a $k$-point  $g\in
G(k)$ consider the map $f_g\colon X\to X\times G$ defined by $x\mapsto
(x,g)$, as above.
Since $x\cdot g=m(x,g)=m(f_g(x))$, we have
$$
g^*\xi=f_g^*\, m^*\,\xi=f^*_g\eta\in H^i(X,A).
$$
By Lemma \ref{lem:constant-map}
$f_g^*\eta$ does not depend on $g$.
Thus $g^*\xi$ does not depend on $g$.
This means that $G(k)$ acts on $H^i(X,A)$ trivially.
\end{proof}


\section{Homogeneous spaces of simply connected groups}

In the proof of the main theorem we shall need a result about strong approximation
in homogeneous spaces of semisimple simply connected groups with connected stabilizers.
If $X = H \backslash G$ is such a homogeneous space, since $G$ is
semisimple and simply connected, the group $\Br(\Gbar)$ is trivial
(see \cite{G}, corollary in the Introduction), hence $\Br_{1,x_0}(X,G)
= \Br_{x_0}(X)$.

\begin{theorem}[{\rm Colliot-Th\'el\`ene and Xu}]\label{prop:simply-connected-homogeneous-spaces}
Let G be a semisimple simply connected $k$-group over a number field $k$,
and let $H\subset G$ be a connected subgroup.
Set $X:=H\backslash G$.
Let $S$ be a non-empty finite set of places of $k$
such that $G(k)$ is dense in $G(\A^S)$.
Then the set of points $x \in  X(\A)$ such that $\langle b, x \rangle = 0$ for all $b \in \Br_{1,x_0}(X,G) = \Br_{x_0}(X)$
coincides with the closure of the set $X(k)\cdot G(k_S)$ in $X(\A)$
for the adelic topology.
\end{theorem}

\begin{proof}
This is very close to a result of Colliot-Th\'el\`ene and Xu, see \cite{CTX}, Theorem 3.7(b).
Since their result is not stated in these terms in \cite{CTX}, we give here a proof
of Theorem \ref{prop:simply-connected-homogeneous-spaces}, following their argument.

We prove the nontrivial inclusion of the theorem.
Let $x \in X(\A)$ be orthogonal to $\Br_{x_0}(X)$.
Then by \cite{CTX}, Theorem 3.3,
there exists $x_1 \in X(k)$ and $g \in G(\A)$ such that $x=x_1\cdot g$.
Let $\sU_X\subset X(\A)$ be an open neighbourhood of $x$.
Clearly there exists an open neighbourhood $\sU_G\subset G(\A)$ of $g$
such that for any $g'\in \sU_G$ we have $x_1\cdot g'\in\sU_X$.
By assumption $G(k)$ is dense in $G(\A^S)$,
hence $G(k)\cdot G(k_S)$ is dense in $G(\A)$.
It follows that there exist $g_0\in G(k)$ and $g_S\in G(k_S)$
such that $g_0 g_S\in \sU_G$,
then $x_1\cdot g_0\cdot g_S\in \sU_X$.
Set $x_2=x_1\cdot g_0$, then $x_2\in X(k)$ and $x_2\cdot g_S\in \sU_X$.
Thus $x$ lies in the closure
of the set $X(k)\cdot G(k_S)$ in $X(\A)$ for the adelic topology.
This proves the nontrivial inclusion.

We prove the trivial inclusion.
Let $x_1\in X(k)$, $g_S\in G(k_S)$ and $b\in \Br_{1,x_0}(X,G)$.
Clearly we have $\langle b,x_1\rangle=0$.
By Corollary \ref{cor:from-compatibility} we have
$\langle b,x_1\cdot g_S\rangle=\langle b,x_1\rangle=0$.
By Lemma \ref{lem:Ducros} below we have $\langle b,x\rangle=0$ for any
$x$ in the closure of the set $X(k)\cdot G(k_S)$.
\end{proof}

\begin{lemma}\label{lem:Ducros}
Let $X$ be a smooth geometrically integral variety over a number field $k$.
Let $b\in\Br(X)$.
Then the function
$$
X(\A)\to\Q/\Z\colon x\mapsto \langle b,x\rangle
$$
is  locally constant in $x$ for the adelic topology in $X(\A)$.
\end{lemma}
\begin{proof}
Arguing as in \cite{San}, Proof of Lemma 6.2, we can reduce our lemma
to the local case.
In other words, it is enough to prove that for any completion $k_v$ of $k$
the function
$$
\phi_b : X(k_v)\to\Br(k_v)\colon x\mapsto b(x)
$$
is  locally constant in $x$.
This follows from non-published results from the thesis of Antoine Ducros,
cf.  \cite{Du}, Part II, Propositions (0.31) and (0.33).

Since those results of Ducros are not published, we give another proof
of this fact. We are grateful to J.-L. Colliot-Th\'el\`ene and to the
referee for this proof. Let $x \in X(k_v)$. The
problem is local, so we can replace $X$ by an affine open subset
containing $x$. From now on, $X$ is assumed to be affine over $k_v$.

To show that the map $\phi_b$ is locally constant around $x$, we may
replace $b$ by $b - b(x) \in \Br(X)$. Now we have $b(x) = 0$ and we
want to prove that $\phi_b$ is zero in a topological neighbourhood of
$x$.

Since $X$ is a smooth affine $k_v$-variety, by a result by Hoobler (see
\cite{Hoo}, Corollary 1) there exists a class $\eta \in H^1(X, \PGL_n)$
such that $b$ is the image of $\eta$ by the usual coboundary map. The
class $\eta$ is represented by an $X$-torsor  $f : Y \to X$ under $\PGL_n$.

For $x'\in X(k_v)$ let $\eta(x')\in H^1(k_v, \PGL_n)$ denote the image of $\eta$
under the  map $(x')^*\colon H^1(X, \PGL_n)\to H^1(k_v, \PGL_n)$.
Consider the exact sequence
$$
1=H^1(k_v,\GL_n)\to H^1(k_v, \PGL_n)\labelto{\Delta} \Br(k_v) \, .
$$
It is clear that $b(x')=\Delta(\eta(x'))$.
From the exact sequence we see that $b(x')=0$ if and only if $\eta(x')=1$.
On the other hand, clearly $\eta(x')$ is the class of the $k_v$-torsor $f^{-1}(x')$ under $\PGL_n$.
It follows that $b(x')=0$ if and only if  $f^{-1}(x')$ contains a $k_v$-point,
i.e.  $x'=f(y)$ for some $y\in Y(k_v)$. Hence the set of points $x' \in X(k_v)$ such that $b(x')= 0$
is exactly the subset $f(Y(k_v))$ of $X(k_v)$.
We now conclude by the implicit function theorem: since $f : Y \to X$ is a smooth
morphism of $k_v$-schemes, the image $f(Y(k_v))$ is an open subset of
$X(k_v)$. Therefore, $\phi_b$ is zero on the open neighbourhood
$f(Y(k_v))$ of $x$, which concludes the proof.
\end{proof}



\section{Proof of the main theorem}
Throughout this section we consider $X = H \backslash G$ satisfying the assumptions of Theorem \ref{thm:Harari-generalized}.
Let $x\in X(\A)$ be an adelic point, we write $x=(x^f,x_\infty)$,
where $x^f\in X(\A^f),\ x_\infty\in X(k_\infty)$.
Let $S$ be a finite set of places of $k$ containing all archimedean places,
and we set $S_f:=S\cap\Omega_f$.
Let $\sU_{X}^f\subset X(\A^f)$ be an open neighbourhood of   $x^f$.
For $v\in\Omega_\infty$, we set $\sU_{X,v}$ to be the connected component of $x_v$ in $X(k_v)$.
We set $\sU_{X,\infty}:=\prod_{v\in \Omega_\infty} \sU_{X,v}$,
then $\sU_{X,\infty}$ is the connected component of $x_\infty$ in $X(k_\infty)$.
We set
$$
\sU_{X}:=\sU_{X}^f\times \sU_{X,\infty}\subset X(\A)
$$
and
$$
\sU'_{X} := \sU_{X}\cdot G\scu(k_{S_f})=\sU_{X}\cdot G\scu(k_S) =
{\sU'}_{X}^f \times \sU'_{X,\infty} \, ,
$$
where ${\sU'}_{X}^f = {\sU}_{X}^f\cdot G\scu(k_{S_f})$
and $\sU'_{X,\infty} =\sU_{X,\infty}  = \sU_{X,\infty}\cdot G\scu(k_{\infty})$
(because $G\scu(k_\infty)$ is a connected topological group, see \cite{OV}, Theorem  5.2.3).
Then $\sU_{X}$ and $\sU'_X$ are  open neighbourhoods of $x$ in $X(\A)$.
We  say that $\sU_{X}$ is
 \emph{the special neighbourhood of $x$ defined by $\sU_{X}^f$}.

For the sake of the argument it will be convenient to introduce
Property (\P) of a pair $(X,G)$:

(\P) \emph{For any point $x\in X(\A)$  orthogonal to $\Br_1(X,G)$,
and for any open neighbourhood $\sU_{X}^f$ of $x^f$,
the set $X(k)\cdot G\scu(k_S)\cap\sU_{X}$
(or equivalently the set $X(k)\cap\sU'_X$,
or equivalently the set $X(k)\cdot G\scu(k_{S_f})\cap\sU_{X}$ )
is non-empty,
where $\sU_{X}$ is the special neighbourhood of $x$ defined by $\sU_{X}^f$.}

  The nontrivial part of Theorem \ref{thm:Harari-generalized} precisely says that property
 (\P)  holds for any $X$, $G$ and $S$ as in the theorem.

We start proving Theorem \ref{thm:Harari-generalized}.
The structure of the proof is somewhat similar to that of Theorem A.1 of \cite{BCTS}.

\medskip
\renewcommand{\ab}{^{\textup{abvar}}}

\begin{subsec}\label{subsec:First-reduction}
\emph{First reduction.}

We reduce Theorem \ref{thm:Harari-generalized} to the case $G\uu=1$.
Let $X$ and $G$ be as in the theorem.
We represent $G$ as an extension
$$
1\to G\lin\to G\to G\ab\to 1,
$$
where $G\lin$ is a connected linear algebraic $k$-group
and $G\ab$ is an abelian variety over $k$.
We use the notation of \ref{subsec:groups}.
Set $G':=G/G\uu$, $Y:=X/G\uu$.
We have a canonical epimorphism $\varphi\colon G\to G'$
and a canonical smooth $\varphi$-equivariant morphism $\psi\colon X\to Y$.
We have  ${G'}\lin=G\lin/G\uu$,
hence ${G'}\uu=1$.
We have ${G'}\ab=G\ab$, hence $\Sha({G'}\ab)$ is finite.

Assume that the pair $(Y,G')$ has Property (\P).
We prove that the pair $(X,G)$ has this property.
Let $x\in X(\A)$ be a point orthogonal to $\Br_1(X,G)$.
Set $y:=\psi(x)\in Y(\A)$. By functoriality, $y$ is orthogonal to $\Br_1(Y,G')$.
Let  $\sU_{X}^f$ be as in (\P), and let $\sU_X$, $\sU'_X$ be the special neighbourhoods of $x$ defined by $\sU_X^f$.
Note that $\sU_X'={\sU_X'}^f\times \sU_{X,\infty}$, where ${\sU_X'}^f$ is an open subset of $X(\A^f)$.
Indeed, $G\scu(k_v)$ is connected for all $v\in \Omega_\infty$ (see \cite{OV}, Theorem  5.2.3).

Set $\sU_{Y}^f:=\psi(\sU_{X}^f)\subset Y(\A^f)$,  $\sU_{Y}:=\psi(\sU_{X})\subset Y(\A)$
and  $\sU_{Y}':=\psi(\sU_{X}')\subset Y(\A)$.
Since $G\uu$ is connected, by Lemma \ref{lem:open-map}
$\sU_{Y}^f$ is open in $Y(\A^f)$ and   $\sU_{Y}$ and $\sU'_{Y}$ are open in $Y(\A)$.
Set $\sU_{Y,v}:=\psi(\sU_{X,v})$.
By \cite{BCTS}, Lemma A.2, for each
$v\in\Omega_\infty$ the set $\sU_{Y,v}$ is the
connected component of $y_v$ in $Y(k_v)$.
Set $\sU_{Y,\infty}:=\prod_{v\in \Omega_\infty}\sU_{Y,v}$.
We have $\sU_Y=\sU_Y^f\times \sU_{Y,\infty}$.
We see that $\sU_{Y}$ is the special neighbourhood of $y$ defined by $\U_Y^f$.
From the split exact sequence
$$
1\to G\uu\to G\scu\to {G'}\sc\to 1 \, ,
$$
we see that the map $G\scu(k_S)\to {G'}\sc(k_S)$ is surjective.
Note that ${G'}\scu={G'}\sc$.
It follows that
$$
\sU_Y'=\sU_Y. G\scu(k_S)=\sU_Y.{G'}\sc(k_S)=\sU_Y.{G'}\scu(k_S).
$$
Since the pair $(Y,G')$ has Property (\P), there exists a $k$-point $y_0 \in Y(k)\cap \sU_Y'$.

Let $X_{y_0}$ denote the fibre of $X$ over $y_0$.
It is a homogeneous space of the unipotent group  $G\uu$.
By Lemma \ref{lem:unipotent-SA}, $X_{y_0}(k)\neq\emptyset$ and $X_{y_0}$
has the strong approximation property away from $\Omega_\infty$:
the set $X_{y_0}(k)$ is dense in $X_{y_0}(\A^f)$.
Consider the set
$\sV^f:=X_{y_0}(\A^f)\cap{\sU'}_{X}^f$.
By Corollary \ref{cor:unipotent-R}, for any $v\in\Omega_\infty$
the set $X_{y_0}(k_v)$ is connected, and by Lemma \ref{lem:unipotent-k}, $X_{y_0}(k_v)$ is one orbit under $G\uu(k_v)$.
Set $\sV:=\sV^f\times X_{y_0}(k_{\infty})$.

Let $v\in\Omega_\infty$.
We show that $X_{y_0}(k_v)\subset\sU_{X,v}$.
Since $y_0\in\sU_Y'=\psi(\sU_X')$, there exists a point $x_v\in\sU_{X,v}$ such that $y_0=\psi(x_v)$.
Clearly $x_v\in X_{y_0}(k_v)$.
Since $X_{y_0}(k_v)$ is one orbit under $G\uu(k_v)$,
we see that $X_{y_0}(k_v)=x_v\cdot G\uu(k_v)\subset \sU_{X,v}$, because $G\uu(k_v)$ is a connected group.
Thus $X_{y_0}(k_\infty)\subset \sU_{X,\infty}$ and
$\sV\subset\sU_X^{\prime f}\times\sU_{X,\infty}=\sU'_X$, hence $\sV\subset X_{y_0}(\A)\cap\sU'_{X}$.

Since  $y_0\in\psi(\sU'_{X})$, the set $\sV$ is non-empty.
Since by Lemma \ref{lem:unipotent-SA} $X_{y_0}(k)$
is dense in $X_{y_0}(\A^f)$,
there is a point $x_0\in X_{y_0}(k)\cap\sV$.
Clearly $x_0\in X(k) \cap \sU'_{X}$.
Thus the pair $(X,G)$ has Property (\P).
We see that in the proof of Theorem \ref{thm:Harari-generalized}
we may assume that $G\uu=1$.
\end{subsec}

\begin{subsec}\label{subsec:Second-reduction}
 \emph{Second reduction.}

By \cite{BCTS}, Proposition 3.1 we may regard $X$ as a homogeneous space
of another connected  group $G'$ such that ${G'}\uu=\{1\}$,  ${G'}\sss$
is semisimple simply connected, and the stabilizers of
the geometric points of $X$ in $G'$ are linear and connected.
We have ${G'}\sc=G\sc$, hence ${G'}\scu=G\scu$ (because $G\scu=G\sc$ and ${G'}\scu={G'}\sc$).
It follows from the construction in the proof of Proposition  3.1 of \cite{BCTS}
that there is a surjective homomorphism $G\ab\to {G'}\ab$.
Since by assumption $\Sha(G\ab)$ is finite, we obtain from
\cite{BCTS}, Lemma A.3 that $\Sha({G'}\ab)$ is finite.

Let us prove that if a point $x \in X(\A)$ is orthogonal to
$\Br_1(X,G)$, then it is orthogonal to $\Br_1(X,G')$.
More precisely, we prove that $\Br_1(X,G')$ is a subgroup of $\Br_1(X,G)$.

By construction (see \cite{BCTS}, proof of Proposition ~3.1),
there is an exact sequence of connected algebraic groups
$$
1 \to S \to G' \xrightarrow{q} G_1 \to 1,
$$
where $G_1$ is the quotient of $G$ by the central subgroup $Z(G) \cap H$ and $S$ is a $k$-torus.
Consider the following natural commutative diagram
\begin{displaymath}
\xymatrix{
& & & X & \\
1 \ar[r] & S \ar[r] & G' \ar[ru]^{\pi'} \ar[r]^q & G_1 \ar[r] \ar[u]^{\pi_1} & 1 \\
& & & G \ar[u]^p \ar@/_2pc/[uu]_(.65){\pi} & \, ,
}
\end{displaymath}
where the maps $\pi$, $\pi'$ and $\pi_1$ are the natural quotient maps.
From this diagram, we deduce the following one, where the second line is exact
(see the top row of diagram \eqref{eq:Sansuc-cor-1}):
\begin{displaymath}
\xymatrix{
& \Br(X) \ar[d]^{\pi_1^*} \ar@/_3pc/[dd]_(.65){\pi^*} \ar[rd]^{{\pi'}^*} & \\
0 = \Pic(\ov{S}) \ar[r] & \Br(\ov{G_1}) \ar[d]^{p^*} \ar[r]^{q^*} & \Br(\ov{G'}) \\
& \Br(\ov{G}) & \, .
}
\end{displaymath}
Therefore, the injectivity of the map $q^* \colon \Br(\ov{G_1}) \to \Br(\ov{G'})$
implies that the natural inclusion $\Br_{1,x_0}(X,G_1) \subset \Br_{1,x_0}(X,G')$ is an equality.
And by functoriality $\Br_{1,x_0}(X,G_1)$ is a subgroup of $\Br_{1,x_0}(X,G)$.

Thus $\Br_{1,x_0}(X,G') = \Br_{1,x_0}(X,G_1)$ is a subgroup of $\Br_{1,x_0}(X,G)$.
It follows that if a point $x \in X(\A)$ is orthogonal to
$\Br_1(X,G)$, then it is orthogonal to $\Br_1(X,G')$.

Thus if Theorem \ref{thm:Harari-generalized} holds for the pair $(X,G')$,
then it holds for $(X,G)$.
We see that we may assume in the proof of Theorem \ref{thm:Harari-generalized}
that  $G\lin$ is reductive, $G\sss$ is  simply connected,
and the stabilizers of the geometric points of $X$ in $G$ are linear and connected.
\end{subsec}

Now in order to prove Theorem  \ref{thm:Harari-generalized} it is enough to prove
the following Theorem \ref{thm:non-connected-Manin}.

\begin{theorem} \label{thm:non-connected-Manin}
Let $k$ be a number field,
$G$ a connected $k$-group, and $X := H \backslash G$ a homogeneous space of $G$ with connected stabilizer $H$.
Assume:

{\rm (i)} $G\uu=\{1\}$,

{\rm (ii)} $H \subset G\lin$, i.e. $H$ is linear,

{\rm (iii)} $G\sss$ is simply connected,

{\rm (iv)} $\Sha(G\ab)$ is finite.

\noindent Let $S\supset\Omega_\infty$ be a finite set of places of $k$
containing  all  archimedean places.
We assume that $G\sc(k)$ is dense in $G\sc(\A^S)$.
Then the pair $(X,G)$ has Property (\P).
\end{theorem}

The homogeneous space $X$ defines a
natural homomorphism $H\tor \to G\sab$.
We first prove a crucial special case of Theorem \ref{thm:non-connected-Manin}.

\begin{proposition} \label{5.4-Manin} With the hypotheses of Theorem
$\ref{thm:non-connected-Manin}$, assume that $H\tor$ {\em injects} into
$G\sab$ (i.e. $H \cap G\sss=H\ssu$), and that the homomorphism $\Br_{1,e}(G\sab) \to \Br_{1,e}(H\tor)$ is surjective.
Then the pair  $(X,G)$ has Property (\P).
\end{proposition}

\begin{construction}\label{subsec:modulo-G-ss}
 Set $Y:=X/G\sss$. Then $Y$ is a homogeneous space of
the semiabelian variety $G\sab$, hence it is a \emph{torsor} of some semiabelian variety $G'$.
 We have ${G'}\ab=G\ab$, hence
$\Sha({G'}\ab)$ is finite.
We have a canonical smooth morphism
$\psi\colon X\to Y$.
\end{construction}

To prove Proposition \ref{5.4-Manin} we need a number of lemmas and propositions.

The following proposition is crucial for our proof of Proposition \ref{5.4-Manin} by d\'evissage.

\begin{proposition}\label{prop : key prop}
Let $G,\ X$ be as in Proposition \ref{5.4-Manin}.
Let $Y,\ \psi\colon X\to Y$ be as in Construction \ref{subsec:modulo-G-ss}.
Let $x_0\in X(k)$, $y_0:=\psi(x_0)$, $X_{y_0}:=\psi^{-1}(y_0)=x_0\cdot G\sss$.
Then the natural pullback homomorphism
$$
i^* : \Br_{1,x_0}(X, G) \to \Br_{x_0}(X_{y_0})
$$
is surjective.
\end{proposition}

Note that $\Br_{1,x_0}(X_{y_0}, G\sss)=\Br_{x_0}(X_{y_0})$ because $\Br(\Gbar\sss)=0$.

\begin{construction}
Consider the map $\pi_{x_0}\colon G\to X,\ g\mapsto x_0\cdot g$.
This map identifies $X$ (resp. $X_{y_0}$)
with a quotient of $G$ (resp. $G\sss$) by a connected subgroup $H'$ (resp. ${H'}\ssu$),
so we have
$$
X=H'\backslash G,\quad X_{y_0}= {H'}\ssu\backslash G\sss.
$$
We define a $k$-variety $Z$ by
$$
Z:={H'}\ssu\backslash G
$$
and denote by $z_0\in Z(k)$ the image of $e\in G(k)$.
  We have a commutative diagram of $k$-varieties:
\begin{displaymath}
\xymatrix{
& 1 \ar[d] & 1 \ar[d] & & 1 \ar[d] & \\
1 \ar[r] & {H'}\ssu \ar[r]^w \ar[rd] \ar[d] & H' \ar[rr]^{v} \ar[d]^{h} & &  H\tor \ar@{-->}[ldd]_(.3){g} \ar[d]^{j} \ar[r] & 1 \\
1 \ar[r] & G\sss \ar[r] \ar[dd]^{\pi'_{x_0}} & G \ar[rr]^{p} \ar[rd]^{\pi_Z} \ar[dd]^{\pi_{x_0}} & & G\sab \ar[dd] \ar[r] & 1 \\
& & & Z \ar[rd]^u \ar[dl]^{q} \ar[ru]_{r} & & \\
& X_{y_0} \ar@{-->}[rru]^(.4){f} \ar[r]^{i} & X \ar[rr]^{\psi} & & Y \ar[d] & \\
& & & & 1 &
}
\end{displaymath}
where the first two rows and the last column are exact sequences of connected algebraic groups,
and the other maps are the natural maps between the different homogeneous spaces.
\end{construction}

The following two lemmas are versions of exact sequence \eqref{eq:Sansuc-3} of Theorem \ref{lem Sansuc}.

\begin{lemma}\label{lem sec1}
The following sequence is exact:
$$
\textup{Br}_{1,e}(G\sab) \xrightarrow{r^*} \textup{Br}_{1,z_0}(Z, G) \xrightarrow{f^*} \Br_{x_0}(X_{y_0}) \rightarrow 0.
$$
\end{lemma}

\begin{proof}
We use the functoriality of  exact sequence \eqref{eq:Sansuc-3} of
Theorem \ref{lem Sansuc} to get the following commutative diagram with exact columns.
Here the second column is the exact sequence \eqref{eq:Sansuc-3} for $Z = {H'}\ssu \backslash G$ and the third column  is  exact sequence \eqref{eq:Sansuc-3}
applied to $X_{y_0} = {H'}\ssu \backslash G\sss$.
\begin{equation}\label{eq:diagram}
  \xymatrix{
    &                                                       & \Pic(G\sss)=0\ar[d]\\
    & \Pic({H'}\ssu) \ar@{=}[r] \ar[d]^{\Delta_{G/Z}} & \Pic({H'}\ssu) \ar[d]_{\Delta_{G\sss / X_{y_0}}}^{\cong} \\
\Br_{1,e}(G\sab) \ar[r]^{r^*} \ar@{=}[d] & \Br_{1,z_0}(Z, G) \ar[r]^{f^*} \ar[d]^{\pi_Z^*} & \Br_{x_0}(X_{y_0})\ar[d] \\
\Br_{1,e}(G\sab) \ar[r]^-{p^*}            & \Br_{1, e}(G)\ar[r]
& \Br_{1,e}(G\sss)=0 \, .
}
\end{equation}
We have  $\Pic(G\sss) = 0$ and $\Br_{1,e}(G\sss) = 0$ by \cite{San}, Lemma 6.9(iv),
because $G\sss$ is simply connected.
From the bottom row of diagram \eqref{eq:Sansuc-cor-1} of  Corollary \ref{cor:Sansuc-cor} we get an exact sequence
$$
\Br_{1,e}(G\sab) \xrightarrow{p^*} \Br_{1, e}(G) \xrightarrow{l^*} \Br_{1,e}(G\sss),
$$
where $l\colon G\sss\to G$ is the canonical embedding.
But
$\Br_{1,e}(G\sss)=0$,
therefore the homomorphism $p^* : \Br_{1,e}(G\sab) \rightarrow \Br_{1, e}(G)$ is surjective.
The composition $r \circ f$ being the trivial morphism, the second row of the diagram is a complex.
A diagram chase in  diagram \eqref{eq:diagram} proves the exactness of the sequence of the lemma.
\end{proof}

\begin{lemma}
\label{lem sec2}
The sequence
$$\Br_{1,x_0}(X, G) \xrightarrow{q^*} \textup{Br}_{1,z_0}(Z, G) \xrightarrow{g^*} \Br_{1, e}(H\tor)$$
is exact.
\end{lemma}

\begin{proof}
We consider the following  diagram,
in which the middle column and the last row are the exact sequences
coming from exact sequence \eqref{eq:Sansuc-3} of Theorem \ref{lem  Sansuc},
and the last column is the exact sequence coming from the exact bottom row of diagram
\eqref{eq:Sansuc-cor-1} of Corollary \ref{cor:Sansuc-cor}:
\begin{eqnarray}
\label{diag Sansuc}
\xymatrix{
                               &                  &\Pic(H')\ar[d]^{w^*}  \\
  \Pic(H') \ar[r]^{w^*} \ar[d]^{\Delta_{G/X}} & \Pic({H'}\ssu) \ar@{=}[r]
  \ar[d]^{\Delta_{G/Z}} & \Pic({H'}\ssu) \ar[d]^{\Delta_{H' / H\tor}} \\
\Br_{1,x_0}(X, G) \ar[r]^{q^*} \ar@{=}[d] & \textup{Br}_{1,z_0}(Z, G) \ar[r]^{g^*} \ar[d]^{\pi_Z^*}
& \Br_{1, e}(H\tor) \ar[d]^{v^*} \\
\Br_{1,x_0}(X, G) \ar[r]^-{\pi_{x_0}^*} & \Br_{1, e}(G) \ar[r]^{h^*} &
\Br_{1, e}(H') \, .
}
\end{eqnarray}
We prove that the diagram is commutative.
In this diagram the two first columns define a commutative diagram by
functoriality, and the second row is a complex.
By construction we have $g \circ v = \pi_Z \circ h$, hence in the diagram we have $v^* \circ g^* = h^* \circ \pi_Z^*$.
Let us prove the commutativity of the square in the top right-hand
corner, i.e. let us prove that $g^* \circ \Delta_{G/Z} = \Delta_{H' / H\tor}$.
We observe that the following diagram of torsors under ${H'}\ssu$ is cartesian:
\begin{displaymath}
\xymatrix{
H' \ar[r]^{v} \ar[d]^{h} & H\tor \ar[d]^{g} \\
G \ar[r]^{\pi_Z} & Z \, ,
}
\end{displaymath}
i.e. $\pi_Z^{-1}(H\tor)=H'$.
In other words, the $H\tor$-torsor $H'$ is the pullback of the
$Z$-torsor $G$ by the morphism $H\tor \xrightarrow{g} Z$.
Therefore, if
$$
1 \rightarrow \Gm \rightarrow H_1 \rightarrow
{H'}^{\textup{ssu}} \rightarrow 1
$$
is a central extension representing an
element $p \in \textup{Pic}({H'}\ssu)$ via the isomorphism
$\textup{Ext}^c_k({H'}\ssu, \mathbb{G}_m) \cong \Pic({H'}\ssu)$,
we get a commutative diagram:
\begin{displaymath}
\xymatrix{
H^1(Z, {H'}\ssu) \ar[r]^{\partial_{H_1}} \ar[d]^{g^*} & H^2(Z, \mathbb{G}_m) \ar[d]^{g^*} \\
H^1(H\tor, {H'}\ssu) \ar[r]^{\partial_{H_1}} & H^2(H\tor, \mathbb{G}_m) \\
}
\end{displaymath}
such that $g^*[G] = [H']$ in $H^1(H\tor, {H'}\ssu)$.
We deduce from
this diagram that $g^*(\partial_{H_1}([G])) = \partial_{H_1}([H'])$ in
$H^2(H\tor, \mathbb{G}_m)$,
i.e. that $\Delta_{H' / H\tor}(p) =
g^*(\Delta_{G/Z}(p))$ in $\Br(H\tor)$.
Therefore the top right-hand square in  diagram (\ref{diag Sansuc}) is commutative.

Returning to  diagram (\ref{diag Sansuc}), we see that its commutativity and the
exactness of the last two columns and of the last row imply,
via an easy diagram chase, that the second row of (\ref{diag Sansuc}) is also exact, hence the sequence of the lemma
is exact.
\end{proof}

For an alternative proof of Lemma \ref{lem sec2} we need
the following generalization of Proposition \ref{lem  Sansuc}:

\begin{proposition}
\label{prop Sansuc general}
  Let $k$ be a field of characteristic zero, and
$$1 \to H_1 \to G \to H_2 \to 1$$
be an exact sequence of connected linear algebraic groups over $k$. Let $\pi : Z \to X$ and
$\pi' : Y \to Z$ be two morphisms of algebraic varieties such that the
composite $Y \to X$ is an $X$-torsor under $G$, such that the
restriction to $H_1$ of the action of $G$ on $Y$ defines the structure
of a $Z$-torsor under $H_1$ on $Y$, and such that $Z \to X$ is a
torsor under $H_2$ via the induced action.
Then we have a natural exact sequence
\begin{equation}\label{eq:Sansuc-general}
\textup{Pic}(Z) \xrightarrow{\varphi_1} \textup{Pic}(H_2) \xrightarrow{\Delta_{Z/X}} \textup{Br}_1(X,Y) \xrightarrow{\pi^*}
\textup{Br}_1(Z,Y) \xrightarrow{\varphi_2'} \textup{Br}_{1,e}(H_2) \, .
\end{equation}
If in addition $z \in Z(k)$, we have an exact sequence
\begin{equation}\label{eq:Sansuc-general-2}
\textup{Pic}(Z) \xrightarrow{i_z^*} \textup{Pic}(H_2) \xrightarrow{\Delta_{Z/X}} \textup{Br}_{1,x}(X,Y) \xrightarrow{\pi^*}
\textup{Br}_{1,z}(Z,Y) \xrightarrow{i_z^*} \textup{Br}_{1,e}(H_2) \, ,
\end{equation}
where $x \in X(k)$ is the image of $z$.
\end{proposition}

\begin{proof}
As in the end of the proof of Theorem \ref{lem Sansuc}, we construct the exact
sequence \eqref{eq:Sansuc-general}
from  the top row of diagram (\ref{eq:Sansuc-1}) applied to the
torsor $Z \to X$:
\begin{equation}\label{eq:Sansuc-1bis}
\textup{Pic}(Z) \xrightarrow{\varphi_1} \textup{Pic}(H_2) \xrightarrow{\Delta_{Z/X}} \textup{Br}(X) \xrightarrow{\pi^*}
\textup{Br}(Z) \xrightarrow{m^* - p_Z^*} \textup{Br}(H_2 \times Z) \, .
\end{equation}
Define a map $\varphi_2'\colon  \textup{Br}(Z,Y) \to \textup{Br}_{1,e}(H_2)$
to be the composition
\begin{align*}
\textup{Br}_1(Z,Y) \labelto{m^*} \textup{Br}_1(Z \times H_2, Y
\times G) \xleftarrow{p_Z^* + p_{H_2}^*} \textup{Br}_{1,e}(H_2, G)
\oplus \textup{Br}_1&(Z,Y) \\
\xrightarrow{\pi_{H_2}}
&\textup{Br}_{1,e}(H_2, G)
\end{align*}
where the morphism $p_Z^* + p_{H_2}^*$ is an isomorphism by Lemma
\ref{lem:product}, and $\pi_{H_2}$ is the projection onto the first factor.
Note that  $\textup{Br}_{1,e}(H_2, G) =\textup{Br}_{1,e}(H_2)$
as a consequence of the injectivity of the homomorphism
$\textup{Br}(\ov{H_2}) \to \textup{Br}(\ov{G})$,
which comes from
the exactness of the top row of diagram \eqref{eq:Sansuc-cor-1} of Corollary \ref{cor:Sansuc-cor}
and the fact that $\textup{Pic}(\ov{G}) \to \textup{Pic}(\ov{H_1})$ is onto (see \cite{San}, Remark 6.11.3).

Consider the diagram
\begin{equation}\label{eq:diagram-BR1Z,Y}
\xymatrix{
\textup{Br}_1(Z,Y) \ar[r]^(.4){m^* - p_Z^*} \ar[d]^{\varphi_2'} & \textup{Br}_1(H_2 \times Z, G
\times Y) \\
\textup{Br}_{1,e}(H_2) \ar[r]^(.35){i_H} & \textup{Br}_{1,e}(H_2) \oplus
\textup{Br}_1(Z,Y) \ar[u]_{p_Z^* + p_{H_2}^*}^{\cong} \, .
}
\end{equation}
This diagram is commutative.
Since $i_H$ is a canonical embedding, we see that $\ker \varphi_2'=\ker(m^*-p_Z^*)$.

We prove that  sequence \eqref{eq:Sansuc-general} is exact.
We use the exactness of \eqref{eq:Sansuc-1bis}.
We know from Theorem \ref{lem Sansuc} that the map $\Delta_{Z/X}\colon\Pic(H_2)\to \Br(X)$
lands in $\Bro(X,Z)$.
Since $\Bro(X,Y)\supset\Bro(X,Z)$, the map $\Delta_{Z/X}\colon\Pic(H_2)\to \Bro(X,Y)$ is defined.
We obtain a commutative diagram with an exact long horizontal line and exact vertical lines:
\begin{equation}\label{eq:Sansuc-1-diagram}
\xymatrix{
                              &                                                &0\ar[d]                        &0\ar[d]           \\
                              &                                                &\Bro(X,Y)\ar[r]^{\pi^*}\ar[d]  &\Bro(Z,Y)\ar[d]   \\
{\Pic}(Z) \ar[r]^{\varphi_1}  &{\Pic}(H_2)\ar[r]^{\Delta_{Z/X}}\ar[ru]^{\Delta_{Z/X}}
                              &{\Br}(X)\ar[r]^{\pi^*}\ar[d]   &{\Br}(Z)\ar[r]^-{m^* - p_Z^*}\ar[d]
            &{\Br}(H_2 \times Z)  \\
                              &
                              &\Br(\Ybar)\ar@{=}[r]
                              &\Br(\Ybar)                         & \, .
}
\end{equation}
Now we see immediately that  sequence \eqref{eq:Sansuc-general} is exact at $\Pic(H_2)$ and $\Bro(X,Y)$.
Since $\ker(m^*-p_Z^*)=\ker \varphi_2'$ in diagram \eqref{eq:diagram-BR1Z,Y}, it follows from the exactness
of \eqref{eq:Sansuc-1bis} at $\Br(Z)$ that  sequence \eqref{eq:Sansuc-general} is exact at $\Bro(Z,Y)$.
Thus \eqref{eq:Sansuc-general} is exact, which completes the proof.
\end{proof}

\begin{subsec}
{\em Alternative proof of Lemma \ref{lem sec2}.}
We consider the exact sequence of linear algebraic groups
$$1 \to {H'}\ssu \to H' \to H\tor \to 1$$
and the torsors $\pi_Z : G \labelto{{H'}\ssu} Z$ and $q : Z \labelto{H\tor} X$.
The composition $\pi_{x_0}\colon G \to X$
is naturally a torsor under $H'$.
Therefore, we can apply Proposition \ref{prop Sansuc  general} to get the  exact sequence
$$
\textup{Pic}(Z) \xrightarrow{g^*} \textup{Pic}(H\tor) \xrightarrow{\Delta_{Z/X}} \textup{Br}_{1,x_0}(X,G) \xrightarrow{q^*}
\textup{Br}_{1,z_0}(Z,G) \xrightarrow{g^*} \textup{Br}_{1,e}(H\tor)\, ,
$$
which concludes the proof.
\qed
\end{subsec}

\begin{subsec}
 {\em Proof of  Proposition \ref{prop : key prop}.}
We have a commutative diagram
\begin{displaymath}
\xymatrix{
& & & 0 \\
\Br_{1,x_0}(X, G) \ar[rd]^{q^*} \ar[rr]^{i^*} & & \Br_{x_0}(X_{y_0}) \ar[ru] & \\
& \Br_{1,z_0}(Z, G) \ar[ru]^{f^*} \ar[rd]^{g^*} & & \\
\Br_{1,e}(G\sab) \ar[ru]^{r^*} \ar[rr]^{j^*} & & \Br_{1,e}(H\tor)
\ar[r] & 0 \, ,
}
\end{displaymath}
where the last row is exact by assumption, and the two slanted sequences are also exact by Lemmas \ref{lem sec1} and \ref{lem sec2}.
A diagram chase shows that the homomorphism
$$
\Br_{1,x_0}(X, G) \xrightarrow{i^*} \Br_{x_0}(X_{y_0})
$$
is surjective, which completes the proof of Proposition \ref{prop : key prop}.
\qed
\end{subsec}

For the proof of Proposition \ref{5.4-Manin}
we need three lemmas.

\begin{lemma}\label{lem:Br(X1)Br(X_2)}
Let $G,X$ be as in Proposition \ref{5.4-Manin} and $Y:=X/G\sss$.
Let $\psi\colon X\to Y$ be the canonical map.
Let $x_1,x_2\in X(k)$, $y_i:=\psi(x_i)$, $X_i:=X_{y_i},\ (i=1,2)$.
Let $r_i\colon\Br_1(X,G)/\Br(k)\to\Br(X_i)/\Br(k)$ be the restriction homomorphisms.
Then there exists a canonical isomorphism $\lambda_{1,2}\colon \Br(X_1)/\Br(k)\isoto\Br(X_2)/\Br(k)$
such that the following diagram commutes:
\begin{equation}\label{eq:diagam-Brauer}
\xymatrix{
\Br_1(X,G)/\Br(k)\ar[r]^{\id}\ar[d]_{r_1}  &\Br_1(X,G)/\Br(k)\ar[d]^{r_2}\\
\Br(X_1)/\Br(k) \ar[r]^{\lambda_{1,2}}
&\Br(X_2)/\Br(k) \, .
}
\end{equation}
\end{lemma}

\begin{proof}
Choose $g\in G(\kbar)$ such that $x_1=x_2\cdot g$.
Then we have $\psi(x_1)=\psi(x_2)\cdot g$, hence
$\psi^{-1}(\psi(x_1))=\psi^{-1}(\psi(x_2))\cdot g$,
thus $\Xbar_1=\Xbar_2\cdot g$.
We obtain commutative diagrams
$$
\xymatrix{
\Xbar_2\ar[r]^g\ar@{_{(}->}[d] &\Xbar_1\ar@{_{(}->}[d]    &&\Br(\Xbar)\ar[r]^{g^*}\ar[d] &\Br(\Xbar)\ar[d] \\
\Xbar\ar[r]^g               &\Xbar
&&\Br(\Xbar_1)\ar[r]^{g^*}     &\Br(\Xbar_2) \, .
}
$$

By Theorem \ref{thm:Brauer}  $g^*\colon \Br(\Xbar)\to\Br(\Xbar)$ is the identity map.
Therefore, the following diagram
$$
\xymatrix{
\Br(X)/\Br(k)\ar[r]^{\id}\ar[d]^{r_1}  &\Br(X)/\Br(k)\ar[d]^{r_2} \\
\Br(X_1)/\Br(k)\ar@{.>}[r]^{\lambda_{1,2}}\ar[d]^{\mu_1}     &\Br(X_2)/\Br(k)\ar[d]^{\mu_2}  \\
\Br(\Xbar_1)      \ar[r]^{g^*}     &\Br(\Xbar_2)
}
$$
is commutative.

Here $r_1$ and $r_2$ are surjective by Proposition \ref{prop : key prop},
 while $\mu_1$ and $\mu_2$ are injective by Lemma \ref{lem:injective} below.
Clearly we can define the dotted arrow (in a unique way)
such that the diagram with this new arrow will be also commutative.
The top square of this new diagram is the desired diagram \eqref{eq:diagam-Brauer}.
\end{proof}

\begin{lemma}\label{lem:injective}
Let $X:=H\backslash G$,
where $G$ is a simply connected semisimple $k$-group over a field $k$ of characteristic 0,
and $H\subset G$ is a connected $k$-subgroup such that $H\tor=1$.
Then $\Br(X)/\Br(k)$ is finite and the canonical homomorphism
$$
\Br(X)/\Br(k)\to\Br(\Xbar)
$$
is injective.
\end{lemma}

\begin{proof}
Let $x_0$ denote the image of $e\in G(k)$ in $X(k)$.
We have a canonical isomorphism $\Br(X)/\Br(k)\cong\Br_{x_0}(X)$.

By Theorem \ref{lem Sansuc} we have a canonical exact sequence
$$
\Pic(G)\to\Pic(H)\to\Br_{1,x_0}(X,G)\to\Br_{1,e}(G),
$$
where by \cite{San}, Lemma 6.9(iv), we have $\Pic(G)=0$ and $\Br_{1,e}(G)=0$.
Moreover, by \cite{G} we have $\Br(\Gbar)=0$, hence $\Br_{1,x_0}(X,G)=\Br_{x_0}(X)$.
We obtain a canonical isomorphism $\Br_{x_0}(X)\cong\Pic(H)$,
functorial in $k$.

We have $H=H\ssu$, hence $\Pic(H)\cong\Pic(H\sss)$. In the commutative diagram
$$
\xymatrix{
\Br(X)/\Br(k)\ar[r]^-{\cong} \ar[d]   & \Pic(H\sss) \ar[d]\\
\Br(\Xbar)\ar[r]^\cong             & \Pic(\ov{H\sss})
}
$$
the right vertical arrow is clearly injective, hence so is the left one.
\end{proof}


\begin{lemma}\label{lem:finite}
Let $X:=H\backslash G$,
where $G$ is a connected $k$-group over a number field $k$,
and $H\subset G$ is a connected linear $k$-subgroup.
Let $M\subset\Br_1(X,G)/\Br(k)$ be a finite subset.
Let $y:=(y_f,y_\infty)\in X(\A^f)\times X(k_\infty)=X(\A)$
be a point orthogonal to $M$ with respect to Manin pairing.
Let $\sU_\infty$ denote the connected component of $y_\infty$ in $X(k_\infty)$.
Then there exists an open neighbourhood $\sU^f\subset X(\A^f)$ of $y_f$
such that  $\sU^f\times\sU_\infty$ is orthogonal to $M$.
\end{lemma}

\begin{proof}
It is an immediate consequence of Lemma \ref{lem:Ducros}, using the
finiteness of $M$.
\end{proof}

\begin{subsec}\label{subsec:M-and-sU}
Recall that $X=H\backslash G$.
Let $x_0$ be the image of $e\in G(k)$ in $X(k)$.
Set $X_0:=x_0\cdot G\sss$.
By assumption $H\cap G\sss=H\ssu$.
By Lemma \ref{lem:injective}  $\Br(X_0)/\Br(k)$ is a finite group.
By Proposition \ref{prop : key prop} the map
$\Br_{1}(X,G)/\Br(k)\to\Br(X_0)/\Br(k)$ is surjective.
We choose a finite subset $M\subset \Br_1(X,G)/\Br(k)$
such that $M$ surjects onto $\Br(X_0)/\Br(k)$.

Now let $x_1\in X(k)$ be any other point.
Set $X_1:=x_1\cdot G\sss$.
It follows from Lemma \ref{lem:Br(X1)Br(X_2)}
that $M\subset \Br_{1}(X,G)/\Br(k)$ surjects onto $\Br(X_1)/\Br(k)$.
\end{subsec}

\begin{subsec}
{\em Proof of Proposition \ref{5.4-Manin}.}
Let $x\in X(\A)$ be orthogonal to $\Br_{1,x_0}(X,G)$.
Let  $\sU_{X}^f$ be an open neighbourhood of the $\A^f$-part $x^f$ of $x$.
Let $\U_X$ be the special neighbourhood of $x$ defined by $\U_X^f$.

Let $M\subset\Br_1(X,G)/\Br(k)$ be as in \ref{subsec:M-and-sU}.
Then $x$ is orthogonal to $M$.
By Lemma \ref{lem:finite} there exists an open neighbourhood
$\U^f$ of $x^f$ such that the corresponding special neighbourhood $\U$ of $x$ in $X(\A)$
is orthogonal to $M$.
We may assume that $\U_X^f\subset \U^f$, then $\U_X\subset \U$, hence $\U_X$ is orthogonal to $M$.

Let $Y$ and $\psi\colon X\to Y$ be as in Construction \ref{subsec:modulo-G-ss}
(i.e. $Y:=X/G\sss$).
Set $y:=\psi(x)\in Y(\A)$.
Since $x$ is orthogonal to $\Br_{1,x_0}(X,G)$, we see by functoriality
that $y$ is orthogonal to $\Br_{1, y_0}(Y,G\sab)$.
Clearly there is a semiabelian variety $G'$ such that $Y$ is a (trivial)
principal homogeneous space of $G'$.
We have a morphism of pairs $(Y, G\sab)\to (Y, G')$,
hence a homomorphism $\Br_1(Y,G')\to\Br_1(Y,G\sab)$.
But $\Br_1(Y,G')=\Br_1(Y)$,
hence $y$ is orthogonal  to the group $\Br_1(Y)$.
As in the  First reduction, see \ref{subsec:First-reduction},
we define $\sU_{Y}^f:=\psi(\sU_{X}^f)$
and we construct the corresponding
special open neighbourhood $\sU_{Y}$ of $y$.
By \cite{BCTS}, Lemma A.2, for any $v\in\Omega_\infty$
we have $\psi(\U_{X,v})=\U_{Y,v}$.
We see that $\psi(\U_X)$ is an open subset
of $Y(\A)$ of the form ${\U}^f\times\U_{\infty}$,
where ${\U}^f\subset Y(\A^f)$ is an open subset
and $\U_{\infty}=\prod_{v\in\Omega_\infty}\U_{Y,v}$, where $\U_{Y,v}\subset Y(k_v)$
is the connected component of $y_v$.
Then $\U_\infty$ is the connected component of $x_\infty$ in $X(k_\infty)$.
Now since $Y$ is a torsor of the semiabelian variety  $G'$
with finite Tate--Shafarevich group,
by \cite{H}, Theorem 4,
there exists a $k$-point $y_1\in Y(k)\cap\psi(\U_X)$.

Let $X_{y_1}$ denote the fibre of $X$ over $y_1$.
Consider the set $\sV:=X_{y_1}(\A)\cap\sU_{X}$,
it is open in $X_{y_1}(\A)$.
Since $y_1\in\psi(\sU_{X})$, the set $\sV$ is non-empty:
there exists a point $x'=(x'_v)\in \sV$.
In particular, $X_{y_1}(k_v)\neq\emptyset$ for any $v\in\Omega_r$.
The variety $X_{y_1}$ is a homogeneous space of $G\sss$ with geometric stabilizer
$\Hbar\cap\Gbar\sss=\Hbar\ssu$.
The group $G\sss$ is semisimple simply
connected by (iii).
The group $\Hbar\ssu$ is connected and character-free, i.e. $(\Hbar\ssu)\tor=1$.
By \cite{B93},  Corollary 7.4, the fact that $X_{y_1}$ has points in all
real completions of $k$ is enough to ensure  that
$X_{y_1}$ has a $k$-point.
Note that $\Br(\Gbar\sss)=0$ by \cite{G},
hence $\Br_1(X_{y_1},G\sss)=\Br(X_{y_1})$.
Since $\U_X$ is orthogonal to $M$, we see that $\sV\subset\U_X$ is orthogonal to $M$.
Since $M$ surjects onto $\Br(X_{y_1})/\Br(k)$, see \ref{subsec:M-and-sU},
we see that $\sV$ is orthogonal to $\Br(X_{y_1})$.
By Theorem \ref{prop:simply-connected-homogeneous-spaces}
(due to Colliot-Th\'el\`ene and Xu)
there is a point of the form $x_1\cdot g_S$ in $\sV$,
where $x_1\in X_{y_1}(k)$ and $g_S\in G\sss(k_S)$.
It follows that the set
$\sV\cdot G\sss(k_S)$ contains a $k$-point of $X_{y_1}$.
Clearly $\sV\cdot G\sss(k_S)\subset \sU_X\cdot G\sss(k_S)$.
Thus $\sU_X\cdot G\sss(k_S)$ contains a $k$-point of $X$,
which shows that the pair $(X,G)$ has Property (\P).
This completes the proof of Proposition \ref{5.4-Manin}.
\qed
\end{subsec}

Let us resume the proof of Theorem \ref{thm:non-connected-Manin}.
We need a construction.

\begin{construction}\label{con:torsor}
We follow an idea of a construction in the proof of \cite{BCTS}, Theorem 3.5.
By Lemma 3 in \cite{CTS77}
there exists a coflasque resolution of the torus $H\tor$, i.e. an exact sequence of $k$-tori
$$
0 \to H\tor \to P \to Q \to 0
$$
where $P$ is a quasi-trivial torus and $Q$ is a coflasque torus.
Recall that a torus is coflasque if for any field extension $K/k$ we have
$H^1(K, \widehat{Q})=0$, where $\widehat{Q}$ is the character group of $Q$.
Consider the $k$-group $F:=G\times P$.
The group $H$ maps diagonally into $F$, and we can consider the quotient homogeneous space
$W := H \backslash F$ of $F$.
There is a natural morphism $t \colon W\to X$.
We have $F\ab=G\ab$, hence $\Sha(F\ab)$ is finite.
We have a canonical homomorphism $H\tor\to F\sab$, and this homomorphism is clearly injective.
Let us prove the following fact, which is necessary to apply Proposition \ref{5.4-Manin} to the homogeneous space $W$ of $F$.
\end{construction}

\begin{lemma}\label{lem:Brauersab}
With the notation of Construction \ref{con:torsor}, the pullback homomorphism
$$
\Br_{1,e}(F\sab) \to \Br_{1,e}(H\tor)
$$
is surjective.
\end{lemma}

\begin{proof}
By definition, we have the following exact commutative diagram:
\begin{equation}
\label{diag tori}
\xymatrix{
& & 0 \ar[d] & 0 \ar[d] & \\
& & G\sab \ar@{=}[r] \ar[d] & G\sab \ar[d] & \\
0 \ar[r] & H\tor \ar[r] \ar@{=}[d] & F\sab = G\sab \times P \ar[r]
\ar[d] & S \ar[d] \ar[r] & 0 \\
0 \ar[r] & H\tor \ar[r] & P \ar[r] \ar[d] & Q \ar[r] \ar[d] &
0 \\
& & 0 & 0 &
}
\end{equation}
where $S$ is defined to be the quotient $F\sab / H\tor$ and all the maps are the natural ones.
The group $S$ is a semi-abelian variety.
By assumption $Q$ is a coflasque torus. Therefore
$$
H^3(k, \widehat{Q}) \cong \prod_{v \textup{ real}} H^3(k_v,
\widehat{Q}) \cong \prod_{v \textup{ real}} H^{1}(k_v, \widehat{Q}) =
0 \, .
$$
Denote by $M_F := [0 \to F\sab]$ (resp. $M_S := [0 \to S]$) the 1-motive
(in degrees $-1$ and $0$) associated to the semi-abelian variety $F\sab$ (resp. $S$), and by ${M_F}^*$ (resp. ${M_S}^*$)
its Cartier dual (see \cite{HSz}, Section 1 page 97 for the definition of the Cartier dual of a 1-motive).
We call a sequence of 1-motives over $k$ exact if the
associated sequence of complexes of fppf sheaves on $\textup{Spec}(k)$
is exact.

Considering diagram (\ref{diag tori}) as an exact diagram in the category of 1-motives over $k$,
we get a commutative exact diagram of 1-motives:
\begin{displaymath}
 \xymatrix{
0 \ar[r] & [0 \to H\tor] \ar[r] \ar@{=}[d] & M_F \ar[r] \ar[d] & M_S \ar[d] \ar[r] & 0 \\
0 \ar[r] & [0 \to H\tor] \ar[r] & [0 \to P] \ar[r] & [ 0 \to Q] \ar[r] & 0 \, .
}
\end{displaymath}
We can dualize this diagram to get the following commutative diagram of 1-motives:
\begin{displaymath}
 \xymatrix{
0 \ar[r] & [\widehat{Q} \to 0] \ar[r] \ar[d] & [\widehat{P} \to 0] \ar[r] \ar[d] & [\widehat{H\tor} \to 0] \ar@{=}[d] \ar[r] & 0 \\
0 \ar[r] & {M_S}^* \ar[r] & {M_F}^* \ar[r] & [\widehat{H\tor} \to 0]
\ar[r] & 0 \, .
}
\end{displaymath}
This diagram is exact as a diagram of complexes of fppf sheaves since the 1-motive $[0 \to H\tor]$ is associated to a $k$-torus
(see \cite{BVB}, Remark 1.3.4).
Hence this exact diagram induces a commutative exact diagram in hypercohomology:
\begin{displaymath}
\xymatrix{
H^2(k, \widehat{P}) \ar[r] \ar[d] & H^2(k, \widehat{H\tor}) \ar[r] \ar@{=}[d] & H^3(k, \widehat{Q}) = 0 \ar[d]  \\
H^1(k, {M_F}^*) \ar[r] & H^2(k, \widehat{H\tor}) \ar[r] & H^2(k, {M_S}^*) \, .
}
\end{displaymath}
Therefore the map $H^1(k, {M_F}^*) \to H^2(k, \widehat{H\tor})$ is surjective.
But by \cite{HSz2}, beginning of Section 4, there are natural maps $\iota_F : H^1(k, {M_F}^*) \to \Br_{1,e}(G\sab)$ and
$\iota_{H\tor} : H^2(k, \widehat{H\tor}) \to \Br_{1,e}(H\tor)$ such that the second map is the canonical isomorphism
of \cite{San}, Lemma 6.9(ii).
 Hence we get a commutative diagram
\begin{displaymath}
\xymatrix{
H^1(k, {M_F}^*) \ar[r] \ar[d]^{\iota_F} & H^2(k, \widehat{H\tor}) \ar[d]_{\cong}^{\iota_{H\tor}} \\
\Br_{1,e}(G\sab) \ar[r] & \Br_{1,e}(H\tor) \, .
}
\end{displaymath}
Since the top map is surjective, so is the bottom one.
\end{proof}

Let $x\in X(\A)$ be a point,
and assume that $x$ is orthogonal to $\Br_{1}(X,G)$.
The map $t \colon W\to X$ is a torsor under a quasi-trivial torus, and we want to lift $x$ to some $w\in W(\A)$
orthogonal to $\Br_1(W,F)$.
To do this, we need the following lemma.

\begin{lemma}\label{lem:torsor-quasi-trivial-exact}
With the above notation, the torsor $t \colon W=H\backslash F\to X$
under the quasi-trivial torus $P$ induces a canonical exact sequence
$$
0 \to \Br_{1,x_0}(X,G) \xrightarrow{t^*} \Br_{1,w_0}(W,F) \xrightarrow{\varphi} \Br_{1,e}(P),
$$
where $x_0$ is the image of $e\in G(k)$ and $w_0$ is the image of $e\in F(k)$.
\end{lemma}

\begin{proof}
We first define the map $\varphi$ of the lemma. The pullback homomorphism:
$\textup{Br}(W) \xrightarrow{\pi_W^*} \textup{Br}(F)$ sends the subgroup $\textup{Br}_{1,w_0}(W, F)$ into $\textup{Br}_{1, e}(F)$.
But $F = G \times P$,
hence thanks to \cite{San}, Lemma 6.6, we have a natural isomorphism $\textup{Br}_{1, e}(F) \cong
\textup{Br}_{1, e}(G) \bigoplus \textup{Br}_{1, e}(P)$.
We compose this map with the second projection
$$
\pi_P \colon \textup{Br}_{1, e}(G) \oplus \textup{Br}_{1, e}(P) \to \textup{Br}_{1, e}(P).
$$
So we define a morphism
$\varphi := \textup{pr}_P \circ \pi_W^* \colon \textup{Br}_{1,w_0}(W, F) \to \textup{Br}_{1, e}(P)$.
The morphism $t^* \colon \textup{Br}_{1,x_0}(X,G) \to \textup{Br}_{1,w_0}(W, F)$ in the lemma
is induced by the morphism of pairs $t \colon (W,F) \to (X,G)$.
By Theorem \ref{lem Sansuc} we have an exact sequence
$$
\Pic(P) \to \Br(X) \xrightarrow{t^*} \Br(W) \, .
$$
The torus $P$ is quasi-trivial, therefore by Lemma 6.9(ii) of \cite{San}, the group $\Pic(P)$ is trivial,
so the homomorphism $t^* : \Br(X) \rightarrow \Br(W)$ is injective.
In particular, the homomorphism $t^* \colon \textup{Br}_{1,x_0}(X,G) \to \textup{Br}_{1,w_0}(W, F)$
in Lemma \ref{lem:torsor-quasi-trivial-exact} is injective.
Therefore, it just remains to prove the exactness of the sequence of the lemma at the term
$\textup{Br}_{1,w_0}(W, F)$.
Consider the following diagram
\begin{displaymath}
\xymatrix{
\Pic(H) \ar@{=}[d] \ar[r]^(.4){\Delta_{G/X}} & \Br_{1,x_0}(X,G) \ar[r] \ar[d]^{t^*}
              & \Br_{1,e}(G) \ar[r] \ar[d] & \Br_{1,e}(H) \ar@{=}[d]\\
\Pic(H) \ar[r]^(.4){\Delta_{F/W}} & \Br_{1,w_0}(W, F) \ar[r] \ar[d]^{\varphi} & \Br_{1, e}(F) \ar[r]
\ar[d] & \Br_{1,e}(H) \\
& \Br_{1,e}(P) \ar@{=}[r] & \Br_{1,e}(P) &
}
\end{displaymath}
where the rows come from Theorem \ref{lem Sansuc}. The
commutativity of this diagram is a consequence of the functoriality of
the exact sequences of Theorem \ref{lem Sansuc} and of the
definition of the map $\varphi$.
We conclude the proof of the exactness of the second column of the
diagram by an easy diagram chase, using the exactness of the two first
rows and that of the third column (see Corollary \ref{cor:Sansuc-cor}).
\end{proof}

\begin{corollary}
With the above notation, if $x \in X(\A)$ is orthogonal to  $\Br_1(X,G)$, then there exists $w \in W(\A)$
such that $t(w) = x$ and $w$ is orthogonal to  $\Br_1(W,F)$.
\end{corollary}

\begin{proof}
Consider the exact sequence of Lemma \ref{lem:torsor-quasi-trivial-exact}.
Taking dual groups, we obtain the dual exact sequence
\begin{equation}\label{eq:Bra-dual}
\Br_{1,e}(P)^D\labelto{\varphi^D}\Br_{1,w_0}(W,F)^D\labelto{t_*} \Br_{1,x_0}(X,G)^D\to 0,
\end{equation}
Let  $m_{X,G}(x)\in \Br_{1,x_0}(X,G)^D$ denote the homomorphism $b\mapsto \langle b,x\rangle\colon \Br_{1,x_0}(X,G)$ $\to\Q/\Z$.
By assumption $m_{G,X}(x)=0$.
We wish to lift $x$ to some $w\in W(\A)$ such that $m_{W,F}(w)=0$.

Since $H^1(k_v,P)=0$ for all $v$, we can lift $x$ to some point $w'\in W(\A)$ such that $t(w')=x$
(we use also Lang's theorem and Hensel's lemma).
Then  $t_*(m_{W,F}(w'))=m_{X,G}(x) \in \Br_{1,x_0}(X,G)^D$.
Since $m_{X,G}(x)=0$, we see from \eqref{eq:Bra-dual}
that $m_{W,F}(w')=\varphi^D(\xi)$ for some $\xi\in\Br_{1,e}(P)^D \cong \Bra(P)^D$.
Let $p\in P(\A)$. By Corollary \ref{cor:compatibility} we have
$$\langle b_W,w'\cdot p\rangle =\langle b_W,w'\rangle + \langle\varphi(b_W),p\rangle$$
for any $b_W\in\Br_{1,w_0}(W,F)$.
This means that
$$m_{W,F}(w'\cdot p)=m_{W,F}(w') + \varphi^D(m_P(p))$$
But we have seen that $m_{W,F}(w')=\varphi^D(\xi)$ for some $\xi\in\Bra(P)^D$.
Now it follows from Lemma \ref{lem:surjectivity} that there exists $p\in P(\A)$ such that $m_P(p)=-\xi$.
Then
$$m_{W,F}(w'\cdot p)=\varphi^D(\xi)+ \varphi^D(-\xi)=0 \, .$$
We set $w := w'\cdot p\in W(\A)$, then $m_{W,F}(w)=0$ and $t(w)=x$.
\end{proof}

\begin{subsec}
We can now resume the proof of Theorem \ref{thm:non-connected-Manin}.
We use Construction \ref{con:torsor}.
 Let $\sU_{X}^f\subset X(\A^f)$ be an open neighbourhood of $x^f$.
Let $\sU_{X}\subset X(\A)$
be the corresponding special neighbourhood of $x$.
Set
$$
\sU_{W}^f:=t^{-1}(\sU_{X}^f)\subset W(\A^f).
$$
For $v\in\Omega_\infty$ let $\sU_{W,v}$ be
the connected component of $w_v$ in $W(k_v)$, then
by \cite{BCTS}, Lemma A.2, we have $t(\sU_{W,v})=\sU_{X,v}$.
Set $\U_{W,\infty}:=\prod_{v\in\Omega_\infty} \sU_{W,v}$.
Set $\U_W:=\U_{W}^f\times\U_{W,\infty}$,
then $\sU_{W}$ is the special open neighbourhood of $w$ defined by $\sU^f_W$,
and  $t(\sU_{W})\subset \sU_{X}$.

The pair $(W,F)$ of $F$ satisfies the hypotheses
of Proposition \ref{5.4-Manin} (see Lemma \ref{lem:Brauersab}),
so by that proposition, there is a point $w_1\in W(k)\cap \sU_{W}\cdot F\sc(k_S)$.
Note that $F\sc=G\sc$.
Set $x_1:=t(w_1)$, then $x_1\in X(k)\cap\sU_{X}\cdot G\sc(k_S)$.
Thus the pair $(X,G)$ has Property (\P).

This completes the proofs of Theorem \ref{thm:non-connected-Manin}
and proves the nontrivial inclusion  of Theorem \ref{thm:Harari-generalized},
that is, that any element of $(X(\A)_{\bullet})^{\Bro(X,G)}$ lies in the closure of the set
$X(k)\cdot G\scu(k_{S_f})$.
The argument in the proof of the trivial inclusion
of Theorem \ref{prop:simply-connected-homogeneous-spaces}
also proves the trivial inclusion of Theorem \ref{thm:Harari-generalized},
that is, that each element of this closure is orthogonal to $\Br_1(X,G)$.
This completes the proof of Main Theorem \ref{thm:Harari-generalized}.
\qed
\end{subsec}



\renewcommand{\ab}{\textup{ab}}

\section{The algebraic Manin obstruction}

In this section we prove Theorem \ref{thm:Main-algebraic} about the algebraic Manin obstruction
(``algebraic'' means coming from $\Bro(X)$).
We  prove this result without using the result of Colliot-Th\'el\`ene and Xu
(Theorem \ref{prop:simply-connected-homogeneous-spaces} or \cite{CTX}, Theorem 3.7(b)).

\begin{subsec}
Before proving Theorem \ref{thm:Main-algebraic}, we need to prove a special case --
an analogue of  Theorem \ref{prop:simply-connected-homogeneous-spaces}.
In \cite{B98}, the first-named author defined, for any connected group $H$
over a field $\kk$ of characteristic 0,
a Galois module $\pi_1(\Hbar)$,
an abelian group $H^1_{\ab}(\kk, H)$ and a canonical abelianization map
$$
{\ab}^1\colon H^1(\kk, H)\to H^1_{\ab}(\kk,H)
$$
(see also \cite{CT}  in any characteristic).
These $\pi_1(\Hbar)$, $H^1_{\ab}(\kk,H)$ and ${\ab}^1$ are functorial in $H$.

Now let $\kk$ be a number field.
Set $\Gamma:=\Gal(\kbar/\kk)$, $\Gamma_v:=\Gal(\kbar_v/\kk_v)$.
We regard $\Gamma_v$ as a subgroup of $\Gamma$.

For $v\in\Omega_f$ we defined in \cite{B98}, Proposition ~4.1(i),  a canonical isomorphism
$$
\lambda_v\colon H^1_{\ab}(\kk_v,H)\isoto (\pi_1(\Hbar)_{\Gamma_v})\tors,
$$
where $\pi_1(\Hbar)_{\Gamma_v}$ denotes the groups of coinvariants
 of $\Gamma_v$ in $\pi_1(\Hbar)$, and $(\  )\tors$ denotes
the torsion subgroup.
Here we set $\lambda'_v:=\lambda_v$.

For $v\in\Omega_\infty$ we defined in \cite{B98}, Proposition  4.2,
 a canonical isomorphism
$$
\lambda_v\colon H^1_{\ab}(\kk_v,H)\isoto H^{-1}(\Gamma_v,\pi_1(\Hbar)).
$$
Here we define a homomorphism $\lambda'_v$ as the composition
$$
\lambda'_v\colon H^1_{\ab}(\kk_v,H)\labelto{\lambda_v} H^{-1}(\Gamma_v,\pi_1(\Hbar))
\into(\pi_1(\Hbar)_{\Gamma_v})\tors\;.
$$
For any $v\in\Omega$ we define the Kottwitz map $\beta_v$ as the composition
$$
\beta_v\colon H^1(\kk_v,H)\labelto{{\ab}^1} H^1_{\ab}(\kk_v,H)\labelto{\lambda'_v}
(\pi_1(\Hbar)_{\Gamma_v})\tors\;.
$$
This map $\beta_v$ is functorial in $H$.
Note that for  $v\in\Omega_f$ the  maps
${\ab}^1\colon H^1(\kk_v,H)\to H^1_{\ab}(\kk_v,H)$ and $\beta_v$  are bijections.
Thus for $v\in\Omega_f$ we have a canonical and functorial in $H$ bijection
$\beta_v\colon H^1(\kk_v,H)\isoto (\pi_1(\Hbar)_{\Gamma_v})\tors$.

For any $v\in\Omega$ we define a map $\mu_v$ as the composition
\begin{equation}\label{eq:mu}
\mu_v\colon H^1(\kk_v,H)\labelto{\beta_v} (\pi_1(\Hbar)_{\Gamma_v})\tors
\labelto{\textup{cor}_v} (\pi_1(\Hbar)_{\Gamma})\tors\;,
\end{equation}
where $\textup{cor}_v$ is the obvious corestriction map.

We write $\bigoplus_v H^1(k_v,H)$ for the set of families $(\xi_v)_{v\in\Omega}$
such that $\xi_v=1$ for almost all $v$.
We define a map
$$
\mu:=\sum_{v\in\Omega}\mu_v\colon \bigoplus_{v\in \Omega} H^1(k_v,H)\to
(\pi_1(\Hbar)_\Gamma)\tors \, .
$$

\begin{proposition}[{\rm Kottwitz \cite{K}, Proposition  2.6, see also \cite{B98}, Theorem 5.15}]
\label{prop:Kottwitz}
 The kernel of the map $\mu$ is equal to the image
of the localization map $H^1(k,H)\to \bigoplus_v H^1(k_v,H)$.
\end{proposition}

\begin{proposition}\label{lem:simply-connected-homogeneous-spaces-orbits}
Let G be a simply connected $k$-group over a number field $k$,
and let $H\subset G$ be a connected geometrically character-free subgroup
(i.e. $H\tor=1$).
Set $X:=H\backslash G$.
Let $S$ be a finite set of places of $k$ containing at least one
nonarchimedean place.
Then any orbit of $G(\A^S)$ in  $X(\A^S)$ contains a $k$-point.
\end{proposition}

\begin{proof}
Write $M:=\pi_1(\Hbar)$.
First we prove that the map
 $$
\mu_S=\sum_{v\in S}\mu_v\colon\prod_{v\in S}
 \ker[H^1(k_v,H)\to H^1(k_v,G)]\to (M_\Gamma)\tors
$$
is surjective.
Since $H$ is geometrically character-free,
the group $M=\pi_1(\Hbar)$ is finite, and therefore
$(M_{\Gamma})\tors=M_{\Gamma}$
and $(M_{\Gamma_v})\tors=M_{\Gamma_v}$.
In this case the map  $\textup{cor}_v$ is the canonical  map
$M_{\Gamma_v}\to M_\Gamma$,
which is clearly surjective.
Let $w\in S$ be a nonarchimedean place, then the map $\beta_w$
 is bijective.
It follows that the map $\mu_w\colon  H^1(k_w,H)\to M_\Gamma$ is surjective
(because $\mu_w=\textup{cor}_w\circ\beta_w$).
Since $w$ is nonarchimedean, we have $H^1(k_w,G)=1$  (because $G$ is simply connected),
hence $\ker[H^1(k_w,H)\to H^1(k_w,G)]=H^1(k_w,H)$.
It follows that the map
$\mu_w\colon \ker [H^1(k_w,H)\to H^1(k_w,G)]\to M_\Gamma$
is surjective.
Now it is clear that the map $\mu_S=\sum_{v\in S}\mu_v$   is surjective.

We prove the proposition.
We must prove that the localization map
$$
X(k)/G(k)\to X(\A^S)/G(\A^S)
$$
is surjective.
In the language of Galois cohomology, we must prove that the localization map
$$
\ker[H^1(k,H)\to H^1(k,G)]\to\bigoplus_{v\notin S}\ker[H^1(k_v,H)\to H^1(k_v,G)]
$$
is surjective.

Let $\xi^S=(\xi_v)\in\bigoplus_{v\notin S}\ker[H^1(k_v,H)\to H^1(k_v,G)]$.
Set $s:=\sum_{v\notin S} \mu_v(\xi_v)\in \pi_1(\Hbar)_\Gamma$.
Since the map $\mu_S$ is surjective,
 there exists an element $\xi_S$ in the product   $\prod_{v\in S} \ker[H^1(k_v,H)\to H^1(k_v,G)]$
such that $\sum_{v\in S}\mu_v(\xi_v)=-s$.
Set
$$
\xi:=(\xi^S,\xi_S)\in \bigoplus_{v\in\Omega}\ker[H^1(k_v,H)\to H^1(k_v,G)]\subset\bigoplus_{v\in\Omega}H^1(k_v,H),
$$
then $\mu(\xi)=\sum_{v\in\Omega}\mu_v(\xi_v)=s+(-s)=0$.
By Proposition \ref{prop:Kottwitz} there exists a class $\xi_0\in H^1(k,H)$
with image $\xi$ in $\bigoplus_{v\in\Omega}H^1(k_v,H)$.
Since the Hasse principle holds for $G$, we have $\xi_0\in \ker[H^1(k,H)\to H^1(k,G)]$.
Since the image of $\xi_0$ in $\bigoplus_{v\in\Omega}H^1(k_v,H)$ is
$(\xi^S,\xi_S)$,
we see that the image of $\xi_0$ in $\bigoplus_{v\notin S}H^1(k_v,H)$ is $\xi^S$.
Thus $\xi^S$ lies in the image of $\ker[H^1(k,H)\to H^1(k,G)]$.
\end{proof}

\begin{theorem}\label{prop:simply-connected-homogeneous-spaces-algebraic}
Let G be a  simply connected $k$-group over a number field $k$,
and let $H\subset G$ be a connected geometrically character-free subgroup
(i.e. $H\tor=1$).
Set $X:=H\backslash G$.
Let $S$ be a finite set of places of $k$ containing at least one nonarchimedean place $v_0$.
Assume that $G\sss(k)$ is dense in $G\sss(\A^S)$.
Then $X$ has strong approximation away from $S$
in the following sense.
Let $x=(x_v)\in X(\A)$ and let $\sU^S\subset X(\A^S)$
be any open neighbourhood of the $\A^S$-part $x^S$ of $x$.
Then there exists a $k$-point $x_0\in X(k)\cap\sU^S$.
Moreover, one can ensure that for $v\in \Omega_\infty\cap S$
the points $x_0$ and $x_v$ lie in the same connected component of $X(k_v)$.
More precisely, there exists $y_{v_0} \in X(k_{v_0})$ such that the point $x' \in X(\A)$
defined by $x'_{v_0} := y_{v_0}$ and $x'_v := x_v$ for $v \neq v_0$ belongs
to the closure of the set $X(k)\cdot G(k_S)$ in $X(\A)$ for the adelic topology.
\end{theorem}

\begin{proof}
Set $\Sigma:= \{v_0\}$.
We denote by $x^\Sigma\in X(\A^\Sigma)$ and $x^S\in x(\A^S)$
the corresponding projections of $x$.
By Proposition \ref{lem:simply-connected-homogeneous-spaces-orbits} applied to the finite set of places $\Sigma$,
there exists a $k$-point $x'_0\in X(k)\cap x^\Sigma\cdot G(\A^\Sigma)$.
Let $y_{v_0} := (x'_0)_{v_0} \in X(k_{v_0})$
and define $x'\in X(\A)$ as in the theorem.
Then there exists $g \in G(\A)$ such that $x'_0 \cdot g = x'$ in $X(\A)$.
Let $\sU \subset X(\A)$ be an open neighbourhood of $x'$.
Since the orbit $x'_0 \cdot G(\A)\subset X(\A)$ is open (because $H$
is connected) and contains $x'$, we may assume that $\sU \subset x'_0 \cdot G(\A)$.

By assumption $G(k).G(k_S)$ is dense in $G(\A)$.
It follows that there
exists $g_0\in G(k)$ and $g_S \in G(k_S)$ such that $x'' := x'_0 \cdot g_0.g_S$ belongs to $\sU$.
Set $x_0:=x'_0\cdot g_0\in X(k)$, then $x''=x_0\cdot g_S$.
We see that $x'' \in X(k) \cdot G(k_S) \cap \sU$.
Therefore, we conclude that $x'$ lies in the closure of  $X(k) \cdot G(k_S)$.

Concerning the infinite places, for $v\in\Omega_\infty\cap S$ we have $x_0\in x_v\cdot G(k_v)$,
because  $x'_0\in x_v\cdot G(k_v)$.
Since $G$ is simply connected, the group $G(k_v)$ is connected (see \cite{OV}, Theorem  5.2.3),
hence the image of $x_0$ in $X(k_v)$ is contained
in the connected component of $x_v$ in $X(k_v)$.
\end{proof}
\end{subsec}

\begin{subsec}{\em Proof of Theorem \ref{thm:Main-algebraic}.}
To prove this theorem, we can follow the proof of Theorem \ref{thm:Harari-generalized} to make  reductions,
so that we may assume the following:

{\rm (i)} $G\uu=\{1\}$,

{\rm (ii)} $H \subset G\lin$, i.e. $H$ is linear,

{\rm (iii)} $G\sss$ is simply connected,

{\rm (iv)} $\Sha(G^{\textup{abvar}})$ is finite.

{\rm (v)} the homomorphism $H\tor \to G\sab$ is injective.
\medskip

Set $\Sigma':=\Omega_\infty\cup\{v_0\}$.
Let $\sU_X^{\Sigma'}\subset X(\A^{\Sigma'})$ be an open neighbourhood of the projection $x^{\Sigma'}\in  X(\A^{\Sigma'})$ of $x$.
Set $\sU_X^f:=\sU_X^{\Sigma'}\times X(k_{v_0})$.
Let $\sU_X$ be the special open neighbourhood of $x$ in $X(\A)$ defined by $\sU_X^f$.
Set $Y:= G\sab / H\tor$, and consider the canonical morphism $\psi : X \to Y$.
Set $y := \psi(x)\in Y(\A)$, then $y$ is orthogonal to the group $\Bro(Y)$ for the Manin pairing.
Hence by \cite{H}, Theorem 4, there exists $y_0 \in Y(k) \cap \psi(\sU_X)$.
Set $X_{y_0}:=\psi^{-1}(y_0)\subset X$ and  $\sV := X_{y_0}(\A) \cap \sU_X$.
Then $\sV$ is open and non-empty since $y_0 \in \psi(\sU_X)$.
As in the proof of Proposition \ref{5.4-Manin}, we know that $X_{y_0}$
is a homogeneous space of the semisimple simply connected group $G\sss=G\sc$,
with connected character-free geometric stabilizers, and with a $k$-point.
Therefore Theorem \ref{prop:simply-connected-homogeneous-spaces-algebraic}
implies that $X_{y_0}(k)\cdot G\sc(k_S) \cap \sV \neq \emptyset$.
In particular, the set $X(k)\cdot G\sc(k_S) \cap \sU_X$ is non-empty.
Set $S':=S\smallsetminus \{v_0\}$, $S'_f:=S'\cap\Omega_f$,
$\sU_X^{\{v_0\}}:=\sU_X^{\Sigma'}\times\sU_{X,\infty}$,
then it follows
that the set $X(k)\cdot G\sc(k_{S'})\cap \sU_X^{\{v_0\}}\subset X(\A^{\{v_0\}})$ is non-empty.
Since $\sU_X^{\{v_0\}}\cdot G\sc(k_{S'})=\sU_X^{\{v_0\}}\cdot G\sc(k_{S'_f})$,
we obtain easily that
the set $X(k)\cdot G\sc(k_{S'_f})\cap \sU_X^{\{v_0\}}\subset X(\A^{\{v_0\}})$ is non-empty.
This completes the proof of Theorem \ref{thm:Main-algebraic}.
\qed

\end{subsec}

\bigskip
\noindent {\bf Acknowledgements.}
The first-named author is grateful to the Tata Institute of Fundamental Research, Mumbai,
where a part of this paper was written, for hospitality
and good working conditions. The second-named author thanks David
Harari for many helpful suggestions. Both authors thank the referee
for his/her comments and suggestions.







\end{document}